\documentclass[a4paper, 11pt]{article}

\usepackage[utf8]{inputenc}
\usepackage[T1]{fontenc}
\usepackage{lmodern}     
\usepackage[english]{babel}
\usepackage{cite, hyphenat, hyperref}

\usepackage{amscd, amsfonts, amsmath}
\usepackage{amssymb, amsthm}
\usepackage{bbm, mathrsfs}
\usepackage{calc}
\usepackage[shortlabels]{enumitem}

\usepackage[a4paper, top=2.5cm, bottom=2.5cm, left=3cm, right=3cm]{geometry}
\allowdisplaybreaks

\numberwithin{equation}{section}
\theoremstyle{plain}

\newtheorem{Lemma}{Lemma}[section]
\newtheorem{Proposition}[Lemma]{Proposition}
\newtheorem{Theorem}[Lemma]{Theorem}
\newtheorem{Corollary}[Lemma]{Corollary}

\theoremstyle{definition}
\newtheorem{Definition}[Lemma]{Definition}
\newtheorem{Example}[Lemma]{Example}
\newtheorem{Remark}[Lemma]{Remark}

\def\R{{\rm R}}
\def\Re{\mathbb{R}}
\def\N{\mathbb{N}}
\DeclareMathOperator*{\esssup}{ess\,sup}

\begin{document}
\title{\textbf{Resolvent and Gronwall inequalities and fixed points of evolution operators}}
\author{Alexander Kalinin\footnote{Department of Mathematics, LMU Munich, Germany. E-mail: {\tt kalinin@math.lmu.de}}}
\date{December 30, 2024}
\maketitle

\begin{abstract}
We introduce kernels and resolvents on preordered sets and derive sharp resolvent inequalities that entail Gronwall inequalities for functions of several variables. In this way, we can prove a fixed point result for operators on topological spaces that extends Banach's fixed point theorem and allows for a wide range of applications.
\end{abstract}


\section{Iterated integral inequalities and resulting fixed points}

The integral inequality for functions of one variable in the classical work~\cite{Gro1919} of Gronwall, known nowadays as \emph{Gronwall's inequality}, is a proven method to analyse the behaviour of solutions to ordinary and stochastic differential equations. A variety of existence and uniqueness results and comparison, growth and stability estimates that play a fundamental role in the solution theories for such equations are based on this crucial inequality.

For a few of the many textbooks on ordinary differential equations (ODEs) that stress this fact, see \cite[Chapter~1]{CodLev55}, \cite[Chapter~3]{Har02} and \cite[Chapters~2 and~4]{Ama90}. Relevant references on stochastic differential equations (SDEs) are \cite[Chapter~4]{IkeWat89}, \cite[Section~5.2]{KarShr91}, \cite[Chapter~9]{RevYor99} and \cite[Chapters~2 and~5]{Mao08}, for instance. As such a practicable tool, the inequality of Gronwall keeps appearing in the literature treating ODEs and SDEs. Slightly refined versions can be found in \cite[Section~1.3]{Pac98}, \cite[p.~128]{Mag23}, \cite[p.~543]{RevYor99} and \cite[p.~45]{Mao08}.

Initially, Gronwall's inequality was enhanced in different directions by Bellman~\cite{Bel43}, Jones~\cite{Jon64}, Willett~\cite{Wil65}, Schmaedeke and Sell~\cite{SchSel68}, Herod~\cite{Her69}, Helton~\cite{Hel69}, Wright, Klasi and Kennebeck~\cite{WriKlaKen71}, Pachpatte~\cite{Pac75} and Groh~\cite{Gro80}. For further initial articles, see the references therein. Versions of Gronwall's inequality for \emph{fractional kernels} with possible singularities were inductively derived by Henry~\cite[pp.~188-189]{Hen81}. Later on, such and other types of \emph{singular kernels} were considered in \cite{YeGaoDin07, Lin13, ZhanWei16, Alm17, Web19, VanCap19, LiuXuZhou24}. As we will see in Examples~\ref{ex:fractional resolvent inequality},~\ref{ex:sums of transformed fractional kernels} and~\ref{ex:fractional kernel}, all these integral inequalities involving kernels are covered by the \emph{resolvent inequality} on preordered sets in Corollary~\ref{co:resolvent inequality}.

Moreover, this result extends the resolvent inequality for functions on compact intervals and continuous kernels by Chu and Metcalf~\cite{ChuMet67} and it entails the \emph{Gronwall inequality for functions of several variables} in Corollary~\ref{co:Gronwall inequalities}. The latter includes the inequalities of Headley~\cite[Theorem~2]{Hea74} and Rasmussen~\cite[Theorem~B]{Ras76} for non-negative functions and linear integral operators. In particular, the scalar version of the deduced inequality by Snow~\cite{Sno72} is a special case. Remarkably, all these results can be described to a considerable extent by a \emph{basic induction principle} for an arbitrary non-empty set $I$.

\begin{Lemma}\label{le:induction principle}
Let $\mathcal{D}$ be a non-empty set of $[0,\infty]$-valued functions on $I$, $\Psi:\mathcal{D}\rightarrow \mathcal{D}$ be an operator and $J$ be a non-empty set in $I$ such that
\begin{equation}\label{eq:induction principle 1}
u \leq \tilde{u}\quad\text{on $J$}\quad\Rightarrow\quad \Psi(u) \leq \Psi(\tilde{u})\quad\text{on $J$}
\end{equation}
for any $u,\tilde{u}\in\mathcal{D}$. Further, let $(u_{n})_{n\in\N_{0}}$ be a sequence in $\mathcal{D}$ satisfying
\begin{equation}\label{eq:induction principle 2}
u_{n} \leq \Psi(u_{n-1})\quad\text{on $J$}\quad\text{for all $n\in\N$.}
\end{equation}
Then the $n$-fold composition of $\Psi$ with itself, denoted by $\Psi^{n}$, satisfies
\begin{equation*}
u_{n} \leq \Psi^{n}(u_{0})\quad\text{on $J$}\quad\text{for any $n\in\N$.}
\end{equation*}
Moreover, equality holds if the inequalities in~\eqref{eq:induction principle 1} and~\eqref{eq:induction principle 2} are equations.
\end{Lemma}

\begin{Remark}
Assume that $u:I\rightarrow [0,\infty]$ satisfies $u(t) \leq \limsup_{n\uparrow\infty} u_{n}(t)$ for all $t\in J$. Then the principle entails that
\begin{equation*}
u \leq \limsup_{n\uparrow\infty} \Psi^{n}(u_{0})\quad\text{on $J$,}
\end{equation*}
and equality holds whenever $u = u_{n}$ on $J$ for every $n\in\N$ and the inequality symbols in~\eqref{eq:induction principle 1} and~\eqref{eq:induction principle 2} are replaced by equal signs.
\end{Remark}

In what follows, let $I$ be endowed with a $\sigma$-field $\mathcal{I}$ and a preorder $\leq$, which is a reflexive and transitive relation, such that the \emph{triangular set} of all $(t,s)\in I\times I$ with $s\leq t$ lies in the product $\sigma$-field $\mathcal{I}\otimes\mathcal{I}$. Then the non-empty sets
\begin{equation*}
I(t):=\{s\in I\,|\,s\leq t\}\quad\text{and}\quad [r,t]:=\{s\in I\,|\, r\leq s\leq t\}
\end{equation*}
belong to $\mathcal{I}$ for any $r,t\in I$ with $r\leq t$, and we notice that Lemma~\ref{le:induction principle} applies to any \emph{nonlinear integral operator} with respect to a $\sigma$-finite measure $\mu$ on $\mathcal{I}$.

\begin{Example}\label{ex:induction principle}
Let $\mathcal{D}$ be the convex cone of all $[0,\infty]$-valued measurable functions on $I$ and $\Psi:\mathcal{D}\rightarrow [0,\infty]^{I}$ be the nonlinear operator defined by
\begin{equation*}
\Psi(u)(t) := v(t) + \bigg(\int_{I(t)}\!F(t,s,u)^{p}\,\mu(\mathrm{d}s)\bigg)^{\frac{1}{p}},
\end{equation*}
where $v\in\mathcal{D}$, $p\geq 1$ and $F:I\times I\times\mathcal{D}\rightarrow [0,\infty]$ is a functional such that $F(\cdot,\cdot,u)$ is product measurable for each $u\in\mathcal{D}$. Then $\Psi$ maps $\mathcal{D}$ into itself, by Fubini's theorem, and if $J$ is a non-empty set in $I$ such that
\begin{equation*}
F(t,s,u) \leq F(t,s,\tilde{u})
\end{equation*}
for all $s,t\in J$ and $u,\tilde{u}\in\mathcal{D}$ with $s\leq t$ and $u\leq\tilde{u}$ on $J$, then the required monotonicity condition~\eqref{eq:induction principle 1} in Lemma~\ref{le:induction principle} is satisfied as soon as $J$ has full measure.
\end{Example}

In this work, we will focus on the case that for the functional $F$ in Example~\ref{ex:induction principle} there is a \emph{non-negative kernel} $k$ on $I$, which is an $[0,\infty]$-valued measurable function on the triangular set, such that
\begin{equation*}
F(t,s,u) = k(t,s)u(s)
\end{equation*}
for all $s,t\in I$ with $s\leq t$ and any measurable function $u:I\rightarrow [0,\infty]$. Based on Fubini's theorem, we recursively define a sequence $(\R_{k,\mu,n})_{n\in\N}$ of non-negative kernels on $I$, which represents the iterated kernels of $k$ relative to $\mu$, via
\begin{equation}\label{eq:resolvent sequence}
\R_{k,\mu,1}(t,s) := k(t,s)\quad\text{and}\quad\R_{k,\mu,n+1}(t,s):=\int_{[s,t]}\!k(t,\tilde{s})\R_{k,\mu,n}(\tilde{s},s)\,\mu(\mathrm{d}\tilde{s}).
\end{equation}
Then the non-negative kernel $\R_{k,\mu}:=\sum_{n=1}^{\infty}\R_{k,\mu,n}$ is called the \emph{resolvent} of $k$ with respect to $\mu$. This name is justified by the fact that the resolvent yields for every $s\in I$ a solution to the generalised linear \emph{Volterra integral equation}
\begin{equation}\label{eq:linear Volterra equation}
\R_{k,\mu}(t,s) = k(t,s) + \int_{[s,t]}\!k(t,\tilde{s})\R_{k,\mu}(\tilde{s},s)\,\mu(\mathrm{d}\tilde{s})
\end{equation}
for all $t\in I$ with $s\leq t$. In particular, if $I$ is a non-degenerate interval in $\Re$, $\mathcal{I}$ is its Borel $\sigma$-field and $\leq$ is the ordinary order, then we recover the standard concepts for kernels and resolvents in the literature on Volterra integral equations, such as \cite[Chapter~9]{GriLonSta90}.

Alternatively, if $\leq$ is the \emph{void order}, that is, $s\leq t$ holds for any $s,t\in I$, then $I\times I$ is the triangular set, we have  $[s,t] = I$ and~\eqref{eq:linear Volterra equation} turns into a linear \emph{Fredholm integral equation}, where $(I,\mathcal{I})$ is merely a measurable space. 

In this general setting, for $p\geq 1$ we inductively derive a \emph{resolvent sequence inequality} in the $L^{p}$-norm with respect to $\mu$ by applying Minkowski's inequality and Fubini's theorem.

\begin{Proposition}\label{pr:resolvent sequence inequality}
Assume that $J$ is a set in $I$ of full measure, $(u_{n})_{n\in\N_{0}}$ is a sequence of $[0,\infty]$-valued measurable functions on $I$ and $v:I\rightarrow [0,\infty]$ is measurable such that
\begin{equation}\label{eq:resolvent sequence inequality 1}
u_{n}(t) \leq v(t) + \bigg(\int_{I(t)}\!k(t,s)^{p}u_{n-1}(s)^{p}\,\mu(\mathrm{d}s)\bigg)^{\frac{1}{p}}
\end{equation}
for all $n\in\N$ and $t\in J$. Then
\begin{equation}\label{eq:resolvent sequence inequality 2}
\begin{split}
u_{n}(t) &\leq v(t)  + \sum_{i=1}^{n-1}\bigg(\int_{I(t)}\!\R_{k^{p},\mu,i}(t,s)v(s)^{p}\,\mu(\mathrm{d}s)\bigg)^{\frac{1}{p}}\\
&\quad +  \bigg(\int_{I(t)}\!\R_{k^{p},\mu,n}(t,s)u_{0}(s)^{p}\,\mu(\mathrm{d}s)\bigg)^{\frac{1}{p}}
\end{split}
\end{equation}
for any $n\in\N$ and $t\in J$, and equality holds if $p = 1$ and the inequality~\eqref{eq:resolvent sequence inequality 1} is an equation.
\end{Proposition}

Under the assumptions of the proposition, if the initial function $u_{0}$ of the sequence $(u_{n})_{n\in\N_{0}}$ happens to satisfy the condition
\begin{equation}\label{eq:resolvent inequality condition 1}
\lim_{n\uparrow\infty}\int_{I(t)}\!\R_{k^{p},\mu,n}(t,s)u_{0}(s)^{p}\,\mu(\mathrm{d}s) = 0
\end{equation}
for each $t\in J$, for which Lemma~\ref{le:sufficient criterion} gives a sufficient criterion, then we may immediately conclude that
\begin{equation*}
\limsup_{n\uparrow\infty} u_{n}(t) \leq v(t) + \sum_{n=1}^{\infty}\bigg(\int_{I(t)}\!\R_{k^{p},\mu,n}(t,s)v(s)^{p}\,\mu(\mathrm{d}s)\bigg)^{\frac{1}{p}}
\end{equation*}
for every $t\in J$. In particular, for $p=1$ this series becomes $\int_{I(t)}\!\R_{k,\mu}(t,s)v(s)\,\mu(\mathrm{d}s)$, by monotone convergence. Hence, we obtain a \emph{resolvent inequality} in the $L^{p}$-norm.

\begin{Corollary}\label{co:resolvent inequality}
Let $J$ be a set in $I$ of full measure and $u,u_{0},v:I\rightarrow [0,\infty]$ be measurable such that $u_{0}(t)\leq u(t)$,
\begin{equation}\label{eq:resolvent inequality 1}
u(t) \leq v(t) + \bigg(\int_{I(t)}\!k(t,s)^{p}u_{0}(s)^{p}\,\mu(\mathrm{d}s)\bigg)^{\frac{1}{p}}
\end{equation}
and condition~\eqref{eq:resolvent inequality condition 1} holds for any $t\in J$. Then
\begin{equation}\label{eq:resolvent inequality 2}
u(t) \leq v(t) + \sum_{n=1}^{\infty}\bigg(\int_{I(t)}\!\R_{k^{p},\mu,n}(t,s)v(s)^{p}\,\mu(\mathrm{d}s)\bigg)^{\frac{1}{p}}
\end{equation}
for every $t\in J$, and equality holds if $p = 1$ and the estimate~\eqref{eq:resolvent inequality 1} is an identity.
\end{Corollary}

If the requirements of the corollary are met, then the estimate~\eqref{eq:resolvent inequality 2} entails that $u(t)$ $\leq v(t) + \esssup_{s\in I(t)}v(s)\mathrm{I}_{k,\mu,p}(t)$ for any $t\in J$, where the measurable series function $\mathrm{I}_{k,\mu,p}:I\rightarrow [0,\infty]$ is given by
\begin{equation}\label{eq:series function}
\mathrm{I}_{k,\mu,p}(t) := \sum_{n=1}^{\infty}\bigg(\int_{I(t)}\!\R_{k^{p},\mu,n}(t,s)\,\mu(\mathrm{d}s)\bigg)^{\frac{1}{p}}.
\end{equation}
In particular, if $v$ and $\mathrm{I}_{k,\mu,p}$ are essentially bounded on $I(t)$ for any $t\in J$, then $u$ satisfies the same property. Similarly, if $v$ and $\mathrm{I}_{k,\mu,p}$ are bounded on $I(t)$ for each $t\in I$, then so is $u$ as soon as $J = I$.

Further, from Corollary~\ref{co:resolvent inequality} we obtain the \emph{$L^{p}$-Gronwall inequalities} in Corollaries~\ref{co:Gronwall sequence inequalities} and~\ref{co:Gronwall inequalities}. Kernels that may fail to satisfy the necessary monotonicity condition~\eqref{eq:monotonicity condition} for this can be handled directly.

\begin{Example}[Fractional $L^{p}$-resolvent inequality]\label{ex:fractional resolvent inequality}
Let $I = I_{1}\times\cdots\times I_{m}$ for some $m\in\N$ and non-degenerate intervals $I_{1},\dots,I_{m}$  in $\Re$ that are bounded from below. Further, let $\mathcal{I}$ be the Borel $\sigma$-field of $I$, the partial order be given by
\begin{equation*}
s\leq t\quad \Leftrightarrow\quad s_{1}\leq t_{1},\dots,s_{m}\leq t_{m}\quad\text{for any $s,t\in I$}
\end{equation*}
and $\mu$ be the Lebesgue measure on $I$. We assume that $k_{0}:I\rightarrow\Re_{+}$ is measurable and increasing, that is, $k_{0}(s)\leq k_{0}(t)$ for all $s,t\in I$ with $s\leq t$, and $\alpha,\beta\in\Re_{+}^{m}$ are such that $\beta_{i} + 1 - \frac{1}{p} < \alpha_{i}$ for each $i\in\{1,\dots,m\}$ and
\begin{equation*}
k(t,s) = k_{0}(t)\prod_{i=1}^{m}(t_{i} - s_{i})^{\alpha_{i} - 1}(s_{i} - t_{i,0})^{-\beta_{i}}
\end{equation*}
for any $s,t\in I$ with $s_{1} < t_{1},\dots,s_{m} < t_{m}$, where $t_{1,0} := \inf I_{1},\dots,t_{m,0} := \inf I_{m}$. Based on Example~\ref{ex:fractional kernel} and Proposition~\ref{pr:kernels on Cartesian products}, the following three assertions hold:
\begin{enumerate}[(1)]
\item Due to Remark~\ref{re:sufficient criterion} and the property~\eqref{eq:limit related to the digamma function} of the gamma function $\Gamma$, the required condition~\eqref{eq:resolvent inequality condition 1} in Corollary~\ref{co:resolvent inequality} is redundant if
\begin{equation}\label{eq:fractional resolvent inequality}
\int_{I(t)}\!u_{0}(s)^{p}\prod_{i=1}^{m}(s_{i} - t_{i,0})^{-\beta_{i}p}\,\mathrm{d}s < \infty\quad\text{for all $t\in J$.}
\end{equation}
For instance, if $t_{1,0}\in I_{1},\dots,t_{m,0}\in I_{m}$, then~\eqref{eq:fractional resolvent inequality} holds if both $\beta = 0$ and $u_{0}$ is locally $p$-fold integrable or $\beta\in [0,\frac{1}{p}[^{m}$ and $u_{0}$ is locally essentially bounded.

\item The number $\alpha_{p,i} := (\alpha_{i} - 1)p + 1$ satisfies $\beta_{i} p < \alpha_{p,i}$ for every $i\in\{1,\dots,m\}$, and the estimate~\eqref{eq:resolvent inequality 2}, the relation~\eqref{eq:fractional kernel 1} and the bound~\eqref{eq:fractional kernel 2} imply that
\begin{align*}
u(t) &\leq v(t)\\
&\quad + c_{\alpha,\beta,p}\sum_{n=1}^{\infty}k_{0}(t)^{n}\bigg(\int_{I(t)}\!v(s)^{p}\prod_{i=1}^{m}\frac{\Gamma(\alpha_{p,i})^{n}(t_{i} - s_{i})^{(\alpha_{p,i} - \beta_{i}p)n +\beta_{i}p - 1}}{\Gamma((\alpha_{p,i} - \beta_{i}p)n + \beta_{i}p)(s_{i} - t_{i,0})^{\beta_{i}p}}\,\mathrm{d}s\bigg)^{\frac{1}{p}} 
\end{align*}
for any $t\in J$ and some $c_{\alpha,\beta,p}  > 0$ that solely depends on $\alpha$, $\beta$ and $p$ such that $c_{\alpha,0,p} = 1$. Further, equality holds if $p = 1$, $k_{0} = 1$, $\beta = 0$ and~\eqref{eq:resolvent inequality 1} is an equation.

\item Under the condition that $\beta\in [0,\frac{1}{p}[^{m}$, we may set $c_{\beta,p} := \prod_{i=1}^{m}\Gamma(1 - \beta_{i}p)^{\frac{1}{p}}$. Then the derived estimate in~(2) cannot exceed the expression
\begin{equation*}
v(t) + \esssup_{s\in I(t)}v(s)c_{\alpha,\beta,p}c_{\beta,p}\sum_{n=1}^{\infty}k_{0}(t)^{n}\prod_{i=1}^{m}\frac{\big(\Gamma(\alpha_{p,i})(t_{i} - t_{i,0})^{\alpha_{p,i} - \beta_{i}p}\big)^{\frac{n}{p}}}{\Gamma((\alpha_{p,i} - \beta_{i}p)n + 1)^{\frac{1}{p}}}
\end{equation*}
for any $t\in I$. Moreover, the appearing series, which agrees with $\mathrm{I}_{k,\mu,p}(t)$ in the case $k_{0} = 1$ and $\beta = 0$, converges absolutely, as it is bounded by the product
\begin{equation*}
\prod_{i=1}^{m}\big(\mathrm{E}_{\alpha_{p,i} - \beta_{i}p,1,p}\big(k_{0}(t)^{\frac{1}{m}}\Gamma(\alpha_{p,i})^{\frac{1}{p}}(t_{i} - t_{i,0})^{\frac{\alpha_{p,i}}{p} - \beta_{i}}\big) - 1\big)
\end{equation*}
with the entire function $\mathrm{E}_{\alpha_{p,i} - \beta_{i}p,1,p}$, where $i\in\{1,\dots,m\}$, that is generally given by~\eqref{eq:extension of the Mittag-Leffler function} and extends the Mittag-Leffler function.
\end{enumerate}

In particular, for $m = p = 1$ this verifies that Corollary~\ref{co:resolvent inequality} yields the inequalities in \cite[p.~188]{Hen81} and \cite[Theorem~1]{YeGaoDin07} and gives a sharper estimate than in \cite[Theorem~3.2]{Web19} under weaker assumptions.
\end{Example}

As a consequence of Corollary~\ref{co:resolvent inequality}, on a non-empty set $X$ we can provide sufficient conditions for a map $\Psi:X\rightarrow X$ to admit at most a unique fixed point by employing distance functions that may take the value infinity or more general type of functions.

Namely, a \emph{premetric} on $X$ is an $[0,\infty]$-valued function on $X\times X$ that satisfies all the properties of a metric except being finite. Such a premetric induces a topology just as a metric and the concept of a metrisable space extends to that of a \emph{premetrisable} one.

A \emph{pseudopremetric} is defined in the same sense by considering a pseudometric instead of a metric, and we shall first assume that $(d_{t})_{t\in I}$ is a family of $[0,\infty]$-valued functions on $X\times X$ such that for any $x,\tilde{x}\in X$ the relation
\begin{equation}\label{eq:pseudometrical condition 1}
x=\tilde{x}\quad\Leftrightarrow\quad d_{t}(x,\tilde{x}) = 0\quad\text{for all $t\in I$}
\end{equation}
holds and the function
\begin{equation}\label{eq:pseudometrical condition 2}
\text{$I\rightarrow [0,\infty],\quad t\mapsto d_{t}(x,\tilde{x})$ is measurable.}
\end{equation}
Thereby, $I$ is merely endowed with the $\sigma$-field $\mathcal{I}$ and the preorder $\leq$, which we interpret as \emph{evolution}, possibly in time, such that $\{(t,s)\in I\times I\,|\,s\leq t\}\in\mathcal{I}\otimes\mathcal{I}$. Let us consider the following regularity condition on $\Psi$ for the \emph{uniqueness of fixed points}:
\begin{enumerate}[label=(C.\arabic*), ref=C.\arabic*, leftmargin=\widthof{(C.1)} + \labelsep]
\item\label{co:1} There are a function $\Lambda:I\times X\times X\rightarrow [0,\infty]$, a $\sigma$-finite measure $\mu$ on $\mathcal{I}$, $p\geq 1$ and a non-negative kernel $\lambda$ on $I$ such that $\Lambda(\cdot,x,\tilde{x})$ is measurable,
\begin{equation*}
d_{t}\big(\Psi(x),\Psi(\tilde{x})\big) \leq \bigg(\int_{I(t)}\!\lambda(t,s)^{p}\Lambda(s,x,\tilde{x})^{p}\,\mu(\mathrm{d}s)\bigg)^{\frac{1}{p}}
\end{equation*}
and $\Lambda(t,x,\tilde{x}) \leq d_{t}(x,\tilde{x})$ for all $t\in I$ and $x,\tilde{x}\in X$.
\end{enumerate}

We observe that if $(d_{t})_{t\in I}$ is a family of pseudopremetrics that is increasing in the sense that $d_{s}\leq d_{t}$ for any $s,t\in I$ with $s\leq t$, then~\eqref{co:1} entails the Lipschitz condition
\begin{equation*}
d_{t}\big(\Psi(x),\Psi(\tilde{x})\big) \leq \lambda_{0}(t)d_{t}(x,\tilde{x})
\end{equation*}
for all $t\in I$ and $x,\tilde{x}\in X$ with the Lipschitz constant $\lambda_{0}(t) := (\int_{I(t)}\!\lambda(t,s)^{p}\,\mu(\mathrm{d}s))^{\frac{1}{p}}$ as soon as $\lambda(t,\cdot)$ is $p$-fold integrable for each $t\in I$.

\begin{Lemma}\label{le:uniqueness of fixed points}
Let~\eqref{co:1} hold and suppose that
\begin{equation}\label{eq:condition for unique fixed points}
\lim_{n\uparrow\infty}\int_{I(t)}\!\R_{\lambda^{p},\mu,,n}(t,s)\Lambda(s,x,\tilde{x})^{p}\,\mu(\mathrm{d}s) = 0
\end{equation}
for any $t\in I$ and $x,\tilde{x}\in \Psi(X)$. Then $\Psi$ admits at most a unique fixed point.
\end{Lemma}

Let us continue with the following regularity condition on $(d_{t})_{t\in I}$ for the \emph{derivation of fixed points} of $\Psi$ as limits of Picard iterations:
\begin{enumerate}[label=(C.\arabic*), ref=C.\arabic*, leftmargin=\widthof{(C.2)} + \labelsep]
\setcounter{enumi}{1}
\item\label{co:2} The function $d_{t}$ is a pseudopremetric for each $t\in I$. 
\end{enumerate}

Under~\eqref{co:2}, we may equip $X$ with the topology of converge with respect to $d_{t}$ for each $t\in I$. That is, a sequence $(x_{n})_{n\in\N}$ in $X$ converges to some $x\in X$ if and only if
\begin{equation}\label{eq:underlying topology}
\lim_{n\uparrow\infty} d_{t}(x_{n},x) = 0\quad\text{for all $t\in I$.}
\end{equation}
By the relation~\eqref{eq:pseudometrical condition 1}, this constitutes the unique topology that is finer than the one induced by $d_{t}$ for any $t\in I$ such that a sequence in $X$ converges if it converges relative to $d_{t}$ for any $t\in I$. Moreover, $X$ turns into a sequential space.

Hence, we shall call $(d_{t})_{t\in I}$ \emph{complete} if any sequence in $X$ that is Cauchy relative to $d_{t}$ for every $t\in I$ converges. For instance, this property holds if $X$ admits a complete premetric $d$ such that every sequence in $X$ that is Cauchy relative to $d_{t}$ for each $t\in I$ is Cauchy with respect to $d$.

We are now in a position to infer a \emph{fixed point result for evolutions operators} from Proposition~\ref{pr:resolvent sequence inequality} that includes an \emph{error estimate} for the appearing Picard iteration.

\begin{Theorem}\label{th:fixed point}
Let~\eqref{co:1} and~\eqref{co:2} be valid and $(d_{t})_{t\in I}$ be complete. Further, let $\Psi$ be sequentially continuous and $x_{0}\in X$ be such that
\begin{equation}\label{eq:fixed point condition}
\sum_{n=1}^{\infty}\bigg(\int_{I(t)}\!\R_{\lambda^{p},\mu,n}(t,s)\Lambda(s,x_{0},\Psi(x_{0}))^{p}\,\mu(\mathrm{d}s)\bigg)^{\frac{1}{p}} < \infty
\end{equation}
for all $t\in I$. Then the Picard sequence $(x_{n})_{n\in\N}$ in $X$ recursively given by $x_{n}:= \Psi(x_{n-1})$ converges to a fixed point $\hat{x}$ of $\Psi$ and satisfies
\begin{equation}\label{eq:fixed point error estimate}
d_{t}(x_{n},\hat{x}) \leq \sum_{i=n}^{\infty}\bigg(\int_{I(t)}\!\R_{\lambda^{p},\mu,i}(t,s)\Lambda(s,x_{0},\Psi(x_{0}))^{p}\,\mu(\mathrm{d}s)\bigg)^{\frac{1}{p}}
\end{equation}
for any $n\in\N$ and $t\in I$. Moreover, if~\eqref{eq:condition for unique fixed points} holds for all $t\in I$ and $x,\tilde{x}\in\Psi(X)$, then $\hat{x}$ is the only fixed point of $\Psi$.
\end{Theorem}

\begin{Example}\label{ex:fixed point thm}
Let $I=\{t\}$ and $d_{t} = d$ for one $t\in I$ and a metric $d$ inducing the topology of $X$ and let $\Psi$ be Lipschitz continuous with Lipschitz constant $\lambda_{0}$. Then~\eqref{co:1} holds for
\begin{equation*}
p = \mu(\{t\}) = 1,\quad \lambda = \lambda_{0}\quad\text{and}\quad \Lambda = d
\end{equation*}
and we recover the geometric series $\R_{\lambda,\mu} = \sum_{n=1}^{\infty} \lambda_{0}^{n}$, which is finite if and only if $\lambda_{0} < 1$. In this case,~\eqref{eq:condition for unique fixed points} and~\eqref{eq:fixed point condition} are always valid for any $t\in I$ and $x,\tilde{x},x_{0}\in X$. So, if 
\begin{equation*}
\text{$d$ is complete and $\lambda_{0} < 1$,}
\end{equation*}
then Theorem~\ref{th:fixed point} reduces to \emph{Banach's fixed point theorem} that originates from \cite{Ban1922}, ensuring a unique fixed point of $\Psi$, with the error bound~\eqref{eq:fixed point error estimate}, which reduces to
\begin{equation*}
d(x_{n},\hat{x})\leq d\big(x_{0},\Psi(x_{0})\big)\frac{\lambda_{0}^{n}}{1-\lambda_{0}}\quad\text{for all $n\in\N$}.
\end{equation*}
\end{Example}

This article is structured as follows. Section~\ref{se:2} is concerned with a detailed analysis of the three main contributions, the resolvent inequalities in Proposition~\ref{pr:resolvent sequence inequality} and Corollary~\ref{co:resolvent inequality} and the fixed point result in Theorem~\ref{th:fixed point}. First, Section~\ref{se:2.1} discusses various properties of kernels and resolvent sequences and considers three types of kernels that can be handled in a general manner.

In Section~\ref{se:2.2} we treat resolvent sequences of sums of kernels and focus on transformed fractional kernels that may admit singularities. In Section~\ref{se:2.3} we deal with kernels on preordered Cartesian products and derive the Gronwall inequalities for functions of several variables in Corollaries~\ref{co:Gronwall sequence inequalities} and~\ref{co:Gronwall inequalities}. Finally, relevant applications of Theorem~\ref{th:fixed point} are indicated in Section~\ref{se:2.4} and the proofs of all the results are contained in Section~\ref{se:3}.

\section{Analysis of the main and consequential results}\label{se:2}

By convention, if $I$ is a non-degenerate interval in $\Re$ with infimum $t_{0}$ and supremum $t_{1}$ and $\leq$ is the ordinary order, then for each monotone function $f:I\rightarrow\Re$ we set $f(t_{0}) := \lim_{t\downarrow t_{0}} f(t)$, if $t_{0}\notin I$, and $f(t_{1}) := \lim_{t\uparrow t_{1}} f(t)$, if $t_{1}\notin I$.

\subsection{Properties and types of kernels and resolvent sequences}\label{se:2.1}

We summarise basic facts on kernels and their resolvent sequences, provide a sufficient criterion for the limit~\eqref{eq:resolvent inequality condition 1} to hold and consider three classes of kernels to which the resolvent inequalities in Proposition~\ref{pr:resolvent sequence inequality} and Corollary~\ref{co:resolvent inequality} directly apply.

First, for an underlying $\sigma$-finite measure $\mu$ on $\mathcal{I}$ and a non-negative kernel $k$ on $I$, the recursive definition~\eqref{eq:resolvent sequence} of the resolvent sequence $(\R_{k,\mu,n})_{n\in\N}$ immediately entails that
\begin{equation*}
\R_{k,\mu,m+n}(t,s) = \int_{[s,t]}\!\R_{k,\mu,m}(t,\tilde{s})\R_{k,\mu,n}(\tilde{s},s)\,\mu(\mathrm{d}\tilde{s})
\end{equation*}
for all $m,n\in\N$ and $s,t\in I$ with $s\leq t$, as can be inductively checked. Hence, by viewing $\R_{k,\mu}$ as pointwise supremum of the sequence $(\sum_{i=1}^{n}\R_{k,\mu,i})_{n\in\N}$ of the partial sums, we deduce from monotone convergence that
\begin{equation*}
\R_{k,\mu}(t,s) = \R_{k,\mu,1}(t,s) + \cdots + \R_{k,\mu,m}(t,s) + \int_{[s,t]}\!\R_{k,\mu,m}(t,\tilde{s})\R_{k,\mu}(\tilde{s},s)\,\mu(\mathrm{d}\tilde{s}).
\end{equation*}
In particular, this justifies that the generalised linear Volterra integral equation~\eqref{eq:linear Volterra equation} is indeed solved by $\R_{k,\mu}(\cdot,s)$ for each $s\in I$.

We readily notice that $\R_{k,\mu,n}$ and $\R_{k,\mu}$, where $n\in\N$, depend superadditively and monotonically on $k$. That is, for another non-negative kernel $l$ on $I$ we have
\begin{equation*}
\R_{k + l,\mu,n}\geq \R_{k,\mu,n} + \R_{l,\mu,n}\quad\text{and}\quad \R_{k+l,\mu}\geq \R_{k,\mu} + \R_{l,\mu},
\end{equation*}
and from $k\leq l$ it follows that $\R_{k,\mu,n} \leq \R_{l,\mu,n}$ and $\R_{k,\mu} \leq \R_{l,\mu}$. If $k$ is of the form $k(t,s)$ $= k_{0}(t)$ for all $s,t\in I$ with $s\leq t$ and some measurable function $k_{0}:I\rightarrow [0,\infty]$ that is increasing or decreasing, then
\begin{equation}\label{eq:basic fact}
\R_{kl,\mu,n}(t,s) \leq \max\{k_{0}(s),k_{0}(t)\}^{n}\R_{l,\mu,n}(t,s)
\end{equation}
for any $s,t\in I$ with $s\leq t$, and equality holds if $k_{0}$ is constant. Similarly, if there is a measurable monotone function $l_{0}:I\rightarrow [0,\infty]$ such that $l(t,s)$ $= l_{0}(s)$ for any $s,t\in I$ with $s\leq t$, then
\begin{equation}\label{eq:basic fact 2}
\R_{kl,\mu,n}(t,s) \leq \R_{k,\mu,n}(t,s)\max\{l_{0}(s),l_{0}(t)\}^{n}
\end{equation}
for all $s,t\in I$ with $s\leq t$, which becomes an equality if $l_{0}$ is actually constant. Moreover, we observe that the \emph{monotonicity condition},
\begin{equation}\label{eq:monotonicity condition}
k(\tilde{s},s) \leq k(t,s)\quad\text{for all $s,\tilde{s},t\in I$ with $s\leq \tilde{s}\leq t$,}
\end{equation}
which we could impose on the kernel $k$, carries over to $\R_{k,\mu,n}$ and $\R_{k,\mu}$. This amounts to $\R_{k,\mu,n}(\tilde{s},s) \leq \R_{k,\mu,n}(t,s)$ and $\R_{k,\mu}(\tilde{s},s) \leq \R_{k,\mu}(t,s)$ for any $s,\tilde{s},t\in I$ with $s\leq\tilde{s}\leq t$.

Based on these considerations, we stress the fact that for $p\geq 1$ and a measurable function $u_{0}:I\rightarrow [0,\infty]$, the required condition~\eqref{eq:resolvent inequality condition 1} in Corollary~\ref{co:resolvent inequality} holds as soon as
\begin{equation}\label{eq:resolvent inequality condition 2}
\sum_{n=1}^{\infty}\bigg(\int_{I(t)}\!\R_{k^{p},\mu,n}(t,s)u_{0}(s)^{p}\,\mu(\mathrm{d}s)\bigg)^{\frac{1}{p}} < \infty
\end{equation}
for all $t\in I$, as otherwise this series cannot be finite. Apparently, this stronger condition is satisfied whenever $u_{0}$ is essentially bounded on $I(t)$ and the series $\mathrm{I}_{k,\mu,p}(t)$ in~\eqref{eq:series function} is finite for each $t\in I$.

In particular, the resolvent inequality applies to the following \emph{tractable class of kernels}, which leads to the Gronwall inequalities in Section~\ref{se:2.3}.

\begin{Definition}\label{de:set of kernels}
Let $K_{\mu}^{p}(I)$ denote the set of all non-negative kernels $k$ on $I$ that satisfy the monotonicity condition~\eqref{eq:monotonicity condition} and for which the series $\mathrm{I}_{k,\mu,p}(t)$ in~\eqref{eq:series function} is finite for each $t\in I$.
\end{Definition}

Based on these concepts, let us investigate the resolvent sequences, the resolvents and the series functions introduced in~\eqref{eq:series function} for three types of kernels.

\begin{Example}[Regular kernels on intervals]\label{ex:regular kernels on intervals}
Let $I$ be a non-degenerate interval in $\Re$, $\mathcal{I}$ be its Borel $\sigma$-field, $\leq$ be the ordinary order and $\mu(\{t\}) = 0$ for all $t\in I$. Then for a non-negative kernel $k$ on $I$ satisfying~\eqref{eq:monotonicity condition} and
\begin{equation*}
\int_{s}^{t}\!k(t,\tilde{s})^{p}\,\mu(\mathrm{d}\tilde{s}) < \infty\quad\text{
for any $s,t\in I$ with $s\leq t$,}
\end{equation*}
it follows inductively from the fundamental theorem of calculus for Riemann-Stieltjes integrals and the associativity of the Lebesgue-Stieltjes integral that
\begin{equation}\label{eq:regular kernels on intervals 1}
\R_{k^{p},\mu,n}(t,s) \leq \frac{k(t,s)^{p}}{(n-1)!}\bigg(\int_{s}^{t}\!k(t,\tilde{s})^{p}\,\mu(\mathrm{d}\tilde{s})\bigg)^{n-1}
\end{equation}
for all $n\in\N$ and $s,t\in I$ with $s\leq t$. The derivation of this estimate is based on the continuity of the integral function $[s,t]\rightarrow\Re_{+}$, $s'\mapsto\int_{s}^{s'}\!k(t,\tilde{s})^{p}\,\mu(\mathrm{d}\tilde{s})$ that is of bounded variation. Hence,
\begin{equation}\label{eq:regular kernels on intervals 2}
\R_{k^{p},\mu}(t,s) \leq k(t,s)^{p}e^{\int_{s}^{t}\!k(t,\tilde{s})^{p}\,\mu(\mathrm{d}\tilde{s})}
\end{equation}
for any $s,t\in I$ with $s\leq t$. Moreover, another application of the fundamental theorem of calculus for Riemann-Stieltjes integrals in combination with monotone convergence yields that
\begin{equation}\label{eq:regular kernels on intervals 3}
\mathrm{I}_{k,\mu,p}(t) \leq \sum_{n=1}^{\infty}\bigg(\frac{1}{n!}\bigg)^{\frac{1}{p}}\bigg(\int_{I(t)}\!k(t,s)^{p}\,\mu(\mathrm{d}s)\bigg)^{\frac{n}{p}}
\end{equation}
for each $t\in I$, and the series on the right-hand side is finite if $\int_{I(t)}\!k(t,s)^{p}\,\mu(\mathrm{ds}) < \infty$. In this case, it equals $\mathrm{E}_{1,1,p}((\int_{I(t)}\!k(t,s)^{p}\,\mu(\mathrm{ds}))^{\frac{1}{p}})$ with the entire function in~\eqref{eq:extension of the Mittag-Leffler function}.

For instance, if $k_{0}:I\rightarrow \Re_{+}$ is increasing and $k_{1}:I\rightarrow [0,\infty]$ is measurable and locally $p$-fold integrable, then $k$ could be of the form
\begin{equation*}
k(t,s) = k_{0}(t)k_{1}(s)\quad\text{for any $s,t\in I$ with $s\leq t$,}
\end{equation*}
and the inequalities~\eqref{eq:regular kernels on intervals 1}-\eqref{eq:regular kernels on intervals 3} become identities whenever $k_{0} = 1$. In conclusion, this shows us that $K_{\mu}^{p}(I)$ consists of all non-negative kernels $k$ on $I$ satisfying~\eqref{eq:monotonicity condition} and $\int_{I(t)}\!k(t,s)^{p}\,\mu(\mathrm{ds}) < \infty$ for any $t\in I$.
\end{Example}

\begin{Example}[The void preorder and kernels of one variable]\label{ex:void preorder}
Let $s\leq t$ always hold for any $s,t\in I$. Then a non-negative kernel $k$ on $I$ satisfies~\eqref{eq:monotonicity condition} if and only if
\begin{equation*}
k(t,s) = k_{0}(s)\quad\text{for any $s,t\in I$}
\end{equation*}
and some measurable function $k_{0}:I\rightarrow [0,\infty]$. In this case, $\R_{k^{p},\mu,n}$ and $\R_{k^{p},\mu}$ are also independent of the first variable $t\in I$ and we have
\begin{equation*}
\R_{k^{p},\mu,n}(\cdot,s) = k_{0}(s)^{p}\bigg(\int_{I}\!k_{0}(t)^{p}\,\mu(\mathrm{d}t)\bigg)^{n-1}
\end{equation*}
for all $n\in\N$ and $s\in I$, which entails that $\R_{k^{p},\mu}(\cdot,s) = k_{0}(s)^{p}\sum_{n=0}^{\infty}(\int_{I}\!k_{0}(t)^{p}\,\mu(\mathrm{d}t))^{n}$. Therefore, for such a kernel $k$ the series function $\mathrm{I}_{k,\mu,p}$ is constant and satisfies
\begin{equation*}
\mathrm{I}_{k,\mu,p} = \sum_{n=1}^{\infty}\bigg(\int_{I}\!k_{0}(t)^{p}\,\mu(\mathrm{d}t)\bigg)^{\frac{n}{p}},
\end{equation*}
and the ratio test shows that $k\in K_{\mu}^{p}(I)$ if and only if $\int_{I}\!k_{0}(t)^{p}\,\mu(\mathrm{d}t) < 1$.
\end{Example}

\begin{Example}[Kernels of submultiplicative type]
Let $k$ be a non-negative kernel on $I$ such that $k(t,\tilde{s})k(\tilde{s},s) \leq k(t,s)$ for any $s,\tilde{s},t\in I$ with $s\leq \tilde{s} \leq t$. Then for any other non-negative kernel $l$ on $I$ we have
\begin{equation*}
\R_{(k l)^{p},\mu,n} \leq k^{p}\R_{l^{p},\mu,n}\quad\text{for all $n\in\N$}\quad\text{and}\quad \R_{(k l)^{p},\mu} \leq k^{p}\R_{l^{p},\mu},
\end{equation*}
which implies that $\mathrm{I}_{kl,\mu,p}(t) \leq \sum_{n=1}^{\infty}(\int_{I(t)}\!k(t,s)^{p}\R_{l^{p},\mu,n}(t,s)\,\mu(\mathrm{d}s))^{\frac{1}{p}}$ for all $t\in I$. Further, if $k$ is of \emph{multiplicative type}, that is, equality holds in the above required property, then all these inequalities become equations.

In particular, let $I$ be a non-degenerate interval in $\Re$, $\mathcal{I}$ be its Borel $\sigma$-field, $\leq$ be the ordinary order, $\mu(\{t\}) = 0$ for any $t\in I$ and $\mu$ be finite on compact sets. Then
\begin{equation*}
\R_{k^{p},\mu,n}(t,s) \leq \frac{k(t,s)^{p}}{(n-1)!}\mu([s,t])^{n-1}\quad\text{and}\quad \R_{k^{p},\mu}(t,s) \leq k(t,s)^{p}e^{\mu([s,t])}
\end{equation*}
for any $n\in\N$ and $s,t\in I$ with $s\leq t$, according to Example~\ref{ex:regular kernels on intervals}. In addition, for the integral function $\hat{k}_{p}:I\rightarrow [0,\infty]$ defined by $\hat{k}_{p}(t) := \int_{I(t)}\!k(t,s)^{p}\,\mu(\mathrm{d}s)$ it follows that
\begin{align*}
\mathrm{I}_{k,\mu,p}(t) &\leq \sum_{n=0}^{\infty}\bigg(\int_{I(t)}\!\frac{k(t,s)^{p}}{n!}\mu([s,t])^{n}\,\mu(\mathrm{d}s)\bigg)^{\frac{1}{p}} \leq \hat{k}_{p}(t)^{\frac{1}{p}}\sum_{n=0}^{\infty}\bigg(\frac{1}{n!}\bigg)^{\frac{1}{p}}\mu\big(I(t)\big)^{\frac{n}{p}}
\end{align*}
for each $t\in I$, and the last estimate is finite if $\mu(I(t))$ and $\hat{k}_{p}(t)$ are. As in the general setting before, all preceding inequalities, except the last one, turn into equalities if $k$ is of multiplicative type.

For instance, this is the case if $\nu$ is a signed Borel measure on $I$ that is finite on compact sets and satisfies $\nu(\{t\}) = 0$ and $k(t,s) = \exp(\nu([s,t]))$ for any $s,t\in I$ with $s\leq t$.
\end{Example}

For a non-negative kernel $k$ on $I$ H\"{o}lder's inequality gives a sufficient criterion for~\eqref{eq:resolvent inequality condition 1} that is weaker than~\eqref{eq:resolvent inequality condition 2} for any $t\in I$ and every measurable function $u_{0}:I\rightarrow [0,\infty]$.

\begin{Lemma}\label{le:sufficient criterion}
Let $(k_{n})_{n\in\N}$ and $(l_{n})_{n\in\N}$ be two sequences of non-negative kernels on $I$ such that for $t\in I$ we have
\begin{equation*}
\R_{k^{p},\mu,n}(t,s) \leq k_{n}(t,s) l_{n}(t,s)\quad\text{for $\mu$-a.e.~$s\in I(t)$}\quad\text{for almost all $n\in\N$.}
\end{equation*}
Further, for each $n\in\N$ let $p_{n}\in [1,\infty]$ and $q_{n}$ be its dual exponent. Then
\begin{equation}\label{eq:sufficient criterion}
\limsup_{n\uparrow\infty}\int_{I(t)}\!\R_{k^{p},\mu,n}(t,s)u_{0}(s)^{p}\,\mu(\mathrm{d}s)\leq \limsup_{n\uparrow\infty} \|k_{n}(t,\cdot)\|_{p_{n},t}\|l_{n}(t,\cdot)u_{0}^{p}\|_{q_{n},t}
\end{equation}
for every measurable function $u_{0}:I\rightarrow [0,\infty]$, where
\begin{equation*}
\|l\|_{q,t} := \bigg(\int_{I(t)}\!l(s)^{q}\,\mu(\mathrm{d}s)\bigg)^{\frac{1}{q}}\quad\text{and}\quad \|l\|_{\infty,t} := \esssup_{s\in I(t)} l(s)
\end{equation*}
for each $q\geq 1$ and any measurable function $l:I(t)\rightarrow [0,\infty]$. In particular, if the estimate in~\eqref{eq:sufficient criterion} vanishes, then~\eqref{eq:resolvent inequality condition 1} is satisfied.
\end{Lemma}

\begin{Remark}\label{re:sufficient criterion}
For a positive kernel $l$ on $I$ we may take $k_{n} = \R_{k^{p},\mu,n}l^{-p}$ and $l_{n} = l^{p}$ for all $n\in\N$. Then~\eqref{eq:resolvent inequality condition 1} holds under each of the following two conditions:
\begin{enumerate}[(1)]
\item $\esssup_{s\in I(t)} l(t,s)u_{0}(s) < \infty$ and $\lim_{n\uparrow\infty}\int_{I(t)}\!\R_{k^{p},\mu,n}(t,s)l(t,s)^{-p}\,\mu(\mathrm{d}s) = 0$, and the latter is redundant if $l = 1$ and the series $\mathrm{I}_{k,\mu,p}(t)$ in~\eqref{eq:series function} is finite.

\item $\int_{I(t)}\!l(t,s)^{p}u_{0}(s)^{p}\,\mu(\mathrm{d}s) < \infty$ and $\lim_{n\uparrow\infty}\esssup_{s\in I(t)}\R_{k^{p},\mu,n}(t,s)l(t,s)^{-p} = 0$.
\end{enumerate}
\end{Remark}

\subsection{Sums of kernels and transformed fractional kernels}\label{se:2.2}

We describe the resolvent sequences of sums of kernels by a recursive scheme and consider, in particular, sums of transformed fractional kernels. As special case, powers of fractional kernels are studied.

First, let $\mu$ be an arbitrary $\sigma$-finite measure on $\mathcal{I}$ and $k$ be a non-negative kernel on $I$ that we write in the form
\begin{equation}\label{eq:sums of kernels}
k = k_{1} + \cdots + k_{N}
\end{equation}
for some $N\in\N$ and non-negative kernels $k_{1},\dots,k_{N}$ on $I$. Then the resolvent sequence $(\R_{k,\mu,n})_{n\in\N}$ of $k$ can be represented as follows.

\begin{Lemma}\label{le:resolvents of sums of kernels}
For each $n\in\N$ let the family $(\R_{k,\mu,n,j})_{j\in\{1,\dots,N\}^{n}}$ of non-negative kernels on $I$ be recursively given by $\R_{k,\mu,1,j}(t,s) := k_{j}(t,s)$ for each $j\in\{1,\dots,N\}$ and
\begin{equation}\label{eq:resolvents of sums of kernels}
\R_{k,\mu,n+1,j}(t,s) := \int_{[s,t]}\!k_{j_{n+1}}(t,\tilde{s})\R_{k,\mu,n,(j_{1},\dots,j_{n})}(\tilde{s},s)\,\mu(\mathrm{d}\tilde{s})
\end{equation}
for all $n\in\N$ and $j\in\{1,\dots,N\}^{n+1}$. Then it holds that $\R_{k,\mu,n,(i,\dots,i)} = \R_{k_{i},\mu,n}$ and $\R_{k,\mu,n} = \sum_{j\in\{1,\dots,N\}^{n}}\R_{k,\mu,n,j}$ for any $n\in\N$ and $i\in\{1,\dots,N\}$.
\end{Lemma}

Let us now restrict our attention to the case that $I$ is a non-degenerate interval in $\Re$ with infimum $t_{0}$, $\mathcal{I}$ is its Borel $\sigma$-field, $\leq$ is the ordinary order and $\mu$ is the Lebesgue measure on $I$. To simplify notation, we shall write
\begin{equation}\label{eq:specific notation}
\R_{k,n},\quad \R_{k,n,j},\quad \R_{k}\quad \text{and}\quad\mathrm{I}_{k,p},\quad\text{where $n\in\N$ and $j\in\{1,\dots,N\}^{n}$,}
\end{equation}
for $\R_{k,\mu,n}$, $\R_{k,\mu,n,j}$, $\R_{k,\mu}$ and $\mathrm{I}_{k,\mu,p}$, respectively. Furthermore, $\Gamma$ denotes the gamma and $\mathrm{B}$ the beta function and we recall the entire general \emph{Mittag-Leffler function} $\mathrm{E}_{\alpha,\beta}:\mathbb{C}\rightarrow\mathbb{C}$, $z\mapsto\sum_{n=0}^{\infty}\Gamma(\alpha n + \beta)^{-1}z^{n}$ to the parameters $\alpha > 0$ and $\beta > 0$ that is discussed in \cite{HauMatSax11}, for instance.

By the Cauchy-Hadamard Theorem, the radius of convergence of the power series that represents $\mathrm{E}_{\alpha,\beta}$ is infinite. In fact, because the digamma function $\psi:]0,\infty[\rightarrow\Re$, $x\mapsto \Gamma'(x)\Gamma(x)^{-1}$ satisfies $\lim_{x\uparrow\infty}\psi(x) = \infty$, L'H{\^o}pital's rule yields that
\begin{equation}\label{eq:limit related to the digamma function}
\lim_{n\uparrow\infty} \Gamma(\alpha n + \beta)^{-\frac{1}{n}} = \lim_{n\uparrow\infty}  e^{-\alpha\psi(\alpha n + \beta)} = 0,\quad\text{even if $\beta = 0$.}
\end{equation}
The Mittag-Leffler function is used in the derivation of $\R_{k}$ and $\mathrm{I}_{k,1}$ in Example~\ref{ex:sums of transformed fractional kernels} when $k_{1},\dots,k_{N}$ are transformed fractional kernels. For the representation~\eqref{eq:sums of transformed fractional kernels 1} of $\R_{k,n}$, where $n\in\N$, we introduce the family of functions below. Thereby, we set
\begin{equation}\label{eq:specific notation 2}
\alpha_{0} := \min_{j=1,\dots,N} \alpha_{j}\quad\text{and}\quad \alpha_{\infty} := \max_{j=1,\dots,N} \alpha_{j}\quad\text{for all $\alpha\in\Re_{+}^{N}$.}
\end{equation}

\begin{Proposition}\label{pr:family of functions}
Let $\alpha\in ]0,\infty[^{N}$, $\beta\in [0,\alpha_{0}[^{N}$ and for each $n\in\N$ let the family $(f_{n,j})_{j\in\{1,\dots,N\}^{n}}$ of $[0,\infty]$-valued measurable functions on $]0,\infty[\times ]0,\infty[$ be recursively given by $f_{1,j}(x,y) := x^{\alpha_{j}-1}y^{-\beta_{j}}$ for every $j\in\{1,\dots,N\}$ and
\begin{equation}\label{eq:family of functions 1}
f_{n+1,j}(x,y) :=  x^{\alpha_{j_{n+1}}}\int_{0}^{1}(1-\lambda)^{\alpha_{j_{n+1}}-1}(\lambda x + y)^{-\beta_{j_{n+1}}}f_{n,(j_{1},\dots,j_{n})}(\lambda x,y)\,\mathrm{d}\lambda
\end{equation}
for any $n\in\N$ and $j\in\{1,\dots,N\}^{n+1}$. Then the following three assertions hold:
\begin{enumerate}[(i)]
\item For every $n\in\N$ the function $f_{n,j}$, where $j\in\{1,\dots,N\}^{n}$, is $]0,\infty[$-valued, decreasing in the second variable and continuous and satisfies
\begin{equation}\label{eq:family of functions 2}
f_{n,j}(x,y) \leq c_{n,j}x^{\beta_{j_{1}} - 1 + \sum_{i=1}^{n} \alpha_{j_{i}} - \beta_{j_{i}}}y^{-\beta_{j_{1}}}
\end{equation}
for all $x,y > 0$, where $c_{n,j} := \prod_{i=1}^{n-1}\mathrm{B}(\alpha_{j_{1}} - \beta_{j_{2}} + \cdots + \alpha_{j_{i}} - \beta_{j_{i+1}},\alpha_{j_{i+1}})$. Thereby, equality holds in~\eqref{eq:family of functions 2} if $\beta_{j_{1}} = \cdots = \beta_{j_{n}} = 0$.

\item For each $n\in\N$ the positive real number $\hat{c}_{n} := \prod_{i=1}^{n-1}\frac{\Gamma((\alpha_{0} - \beta_{\infty})i)}{\Gamma((\alpha_{0} - \beta_{\infty})i + \beta_{\infty})}$ satisfies
\begin{equation}\label{eq:family of functions 3}
c_{n,j}\leq \frac{\hat{c}_{n}\Gamma(\alpha_{0})^{n}}{\Gamma((\alpha_{0} - \beta_{\infty})n + \beta_{\infty})}\quad\text{for any $j\in\{1,\dots,N\}^{n}$,}
\end{equation}
and equality holds if $N = 1$. Further, $\max_{n\in\N}\hat{c}_{n} = \hat{c}$ for the global minimum point $x_{\Gamma}$ of $\Gamma$, $n_{\Gamma} := \min\{i\in\N\,|\,(\alpha_{0}-\beta_{\infty})i \geq x_{\Gamma}\}$ and $\hat{c} := \max_{i=1,\dots,n_{\Gamma}}\hat{c}_{i}$.

\item The series function $f := \sum_{n=1}^{\infty}\sum_{j\in\{1,\dots,N\}^{n}} f_{n,j}$ takes all its values in $]0,\infty[$, is decreasing in the second variable and continuous and satisfies
\begin{equation}\label{eq:family of functions 4}
f(x,y) \leq N\Gamma(\alpha_{0})\hat{c}\frac{ \varphi_{\alpha_{0},\alpha_{\infty}}(x)}{x\varphi_{\beta_{\infty},\beta_{0}}(y)}\mathrm{E}_{\alpha_{0} - \beta_{\infty},\alpha_{0}}\big(N\Gamma(\alpha_{0})\varphi_{\alpha_{0} - \beta_{\infty},\alpha_{\infty} - \beta_{0}}(x)\big)
\end{equation}
for all $x,y > 0$ with $\varphi_{\gamma,\delta}:]0,\infty[\rightarrow]0,\infty[$, $x\mapsto x^{\gamma}\mathbbm{1}_{]0,1]}(x) + x^{\delta}\mathbbm{1}_{]1,\infty[}(x)$ for $\gamma,\delta > 0$. Finally, we have equality in~\eqref{eq:family of functions 4} if $N = 1$ and $\beta = 0$.
\end{enumerate}
\end{Proposition}

\begin{Example}\label{ex:family of functions}
Bearing in mind that $\hat{c}_{n} = \hat{c} = 1$ for all $n\in\N$ whenever $\beta = 0$, let us consider the following two special cases in the setting of Proposition~\ref{pr:family of functions}:
\begin{enumerate}[(1)]
\item For $N = 1$ we set $f_{n} := f_{n,j}$ and $c_{n} := c_{n,j}$, where $j\in\{1\}^{n}$. Then the derived estimates~\eqref{eq:family of functions 2} and~\eqref{eq:family of functions 4} reduce to
\begin{equation*}
f_{n}(x,y) \leq c_{n}x^{(\alpha - \beta) n  + \beta - 1}y^{-\beta}\quad\text{and}\quad 
f(x,y) \leq \Gamma(\alpha)\hat{c}x^{\alpha - 1}y^{-\beta}\mathrm{E}_{\alpha - \beta,\alpha}(\Gamma(\alpha)x^{\alpha - \beta})
\end{equation*}
for any $n\in\N$ and $x,y > 0$, where $c_{n} = \frac{\hat{c}_{n}\Gamma(\alpha)^{n}}{\Gamma((\alpha - \beta)n + \beta)}$. Moreover, the inequalities turn into identities for $\beta = 0$.

\item For $\beta = 0$ equality in~\eqref{eq:family of functions 2} holds. Namely, $f_{n,j}(x,\cdot) = c_{n,j}x^{\alpha_{j_{1}} + \cdots + \alpha_{j_{n}} - 1}$ for all $n\in\N$, $j\in\{1,\dots,N\}^{n}$ and $x > 0$ with $c_{n,j}=\frac{\Gamma(\alpha_{j_{1}})\cdots\Gamma(\alpha_{j_{n}})}{\Gamma(\alpha_{j_{1}} + \cdots + \alpha_{j_{n}})}$. This entails that
\begin{equation*}
\sum_{j\in \{1,\dots,N\}^{n}}f_{n,j}(x,\cdot) = \sum_{\substack{i\in\{0,\dots,n\}^{N}:\\
i_{1} + \cdots + i_{N} = n}}\frac{n!}{i_{1}!\cdots i_{N}!}\Gamma(\alpha_{1})^{i_{1}}\cdots\Gamma(\alpha_{N})^{i_{N}}\frac{x^{\langle i,\alpha\rangle - 1}}{\Gamma(\langle i,\alpha\rangle)}
\end{equation*}
for any $n\in\N$ and $x > 0$, where $\langle\cdot,\cdot\rangle$ is the inner product on $\mathbb{R}^{N}$. This formula is based on the fact that
\begin{equation*}
\frac{n!}{i_{1}!\cdots i_{N}!}
\end{equation*}
is the amount of all $j\in\{1,\dots,N\}^{n}$ such that $|h\in\{1,\dots,n\}\,|\,j_{h} = m\}| = i_{m}$ for any $m\in\{1,\dots,N\}$, where $i\in\{0,\dots,n\}^{N}$ satisfies $i_{1} + \cdots + i_{N} = n$  
\end{enumerate}
\end{Example}

Based on these preliminaries and the notations~\eqref{eq:specific notation} and~\eqref{eq:specific notation 2}, we consider the subsequent type of kernels $k_{1},\dots,k_{N}$ in~\eqref{eq:sums of kernels}.

\begin{Example}[Sums of transformed fractional kernels]\label{ex:sums of transformed fractional kernels}
Let $\alpha\in ]0,\infty[^{N}$, $\beta\in [0,\alpha_{0}[^{N}$ and $\varphi:I\rightarrow\Re$ be strictly increasing and locally absolutely continuous such that $\varphi(t_{0}) > - \infty$ and
\begin{equation*}
k_{j}(t,s) = \dot{\varphi}(s)\big(\varphi(t) - \varphi(s)\big)^{\alpha_{j} - 1}\big(\varphi(s) - \varphi(t_{0})\big)^{-\beta_{j}}
\end{equation*}
for any $j\in\{1,\dots,N\}$ and $s,t\in I$ with $t_{0} < s < t$, where $\dot{\varphi}$ is a positive weak derivative of $\varphi$. For instance, if $I=\Re_{+}$, then for any $\gamma > 0$ we may take the function
\begin{equation*}
\varphi(t) = t^{\gamma}\quad\text{for all $t\geq 0$.}
\end{equation*}
In the general case, for each $n\in\N$ the family $(f_{n,j})_{j\in\{1,\dots,N\}^{n}}$ of $]0,\infty[$-valued continuous functions on $]0,\infty[\times ]0,\infty[$ in Proposition~\ref{pr:family of functions} satisfies \begin{equation*}
\R_{k,n,i}(t,s) = \dot{\varphi}(s)f_{n,i}\big(\varphi(t) - \varphi(s),\varphi(s) - \varphi(t_{0})\big)
\end{equation*}
and
\begin{equation}\label{eq:sums of transformed fractional kernels 1}
\R_{k,n}(t,s) = \dot{\varphi}(s)\sum_{j\in \{1,\dots,N\}^{n}}f_{n,j}\big(\varphi(t)-\varphi(s),\varphi(s) - \varphi(t_{0})\big)
\end{equation}
for all $i\in\{1,\dots,N\}^{n}$ and $s,t\in I$ with $t_{0} < s < t$, as an induction proof and Lemma~\ref{le:resolvents of sums of kernels} show. So, for the $]0,\infty[$-valued continuous function $f := \sum_{n=1}^{\infty}\sum_{j\in\{1,\dots,N\}^{n}} f_{n,j}$ we have
\begin{equation*}
\R_{k}(t,s) = \dot{\varphi}(s)f\big(\varphi(t) - \varphi(s),\varphi(s) - \varphi(t_{0})\big)
\end{equation*}
for any $s,t\in I$ with $t_{0} < s < t$, by the definition of the resolvent. Moreover, for every measurable function $u_{0}:I\rightarrow [0,\infty]$ condition~\eqref{eq:resolvent inequality condition 1} is satisfied when $p = 1$ if
\begin{equation*}
\int_{I(t)}\!\dot{\varphi}(s)(\varphi(s) - \varphi(t_{0}))^{-\beta_{j}}u_{0}(s)\,\mathrm{d}s < \infty
\end{equation*}
for all $j\in\{1,\dots,N\}$ and $t\in I$, according to Lemma~\ref{le:sufficient criterion}. In fact, the two estimates~\eqref{eq:family of functions 2} and~\eqref{eq:family of functions 3} in Proposition~\ref{pr:family of functions} entail that
\begin{equation}\label{eq:sums of transformed fractional kernels 2}
\R_{k,n}(t,s) \leq \frac{\hat{c}\big(N\Gamma(\alpha_{0})\big)^{n}}{\Gamma((\alpha_{0} - \beta_{\infty})n + \beta_{\infty})}k_{n}(t,s)l(s)
\end{equation}
for any $n\in\N$ and $s,t\in I$ with $t_{0} < s < t$, where $l := \max_{j\in\{1,\dots,N\}}\dot{\varphi}(\varphi - \varphi(t_{0}))^{-\beta_{j}}$ is measurable and the non-negative kernel $k_{n}$ on $I$ is defined via
\begin{equation*}
k_{n}(t,s) := \max_{j\in\{1,\dots,N\}^{n}}(\varphi(t) - \varphi(s))^{\beta_{j_{1}} - 1 + \sum_{i=1}^{n}\alpha_{j_{i}}- \beta_{j_{i}}},
\end{equation*}
and equality holds in~\eqref{eq:sums of transformed fractional kernels 2} if $N = 1$ and $\beta = 0$. Thus, all requirements of Lemma~\ref{le:sufficient criterion} are met, because~\eqref{eq:limit related to the digamma function} is valid and we have
\begin{equation*}
\sup_{s\in I(t)}k_{n}(t,s) \leq \max\{1,\varphi(t) - \varphi(t_{0})\}^{(\alpha_{\infty} - \beta_{0})n + \beta_{0} - 1}
\end{equation*}
for each $t\in I$ for almost all $n\in\N$. Finally, the series $\mathrm{I}_{k,1}(t)$, which we introduced in~\eqref{eq:specific notation}, reduces to $\int_{I(t)}\!\R_{k}(t,s)\,\mathrm{d}s$ and equals the expression
\begin{equation*}
(\varphi(t) - \varphi(t_{0}))\int_{0}^{1}\!f\big((1-\lambda)(\varphi(t) - \varphi(t_{0})),\lambda(\varphi(t) - \varphi(t_{0}))\big)\,\mathrm{d}\lambda
\end{equation*}
for any $t\in I$ with $t > t_{0}$, by a substitution and monotone convergence. Thus, under the condition that $\beta_{\infty} < 1$, we see that $\mathrm{I}_{k,1}(t) < \infty$ and the Mittag-Leffler function satisfies
\begin{equation*}
\mathrm{I}_{k,1}(t) \leq \hat{c}\Gamma(1 - \beta_{\infty})\big(\mathrm{E}_{\alpha_{0} - \beta_{\infty},1}\big(N\Gamma(\alpha_{0})\varphi_{\alpha_{0} - \beta_{\infty},\alpha_{\infty} - \beta_{0}}(\varphi(t) - \varphi(t_{0}))\big) - 1 \big),
\end{equation*}
and equality holds if $N = 1$ and $\beta = 0$. By the facts~\eqref{eq:basic fact} and~\eqref{eq:basic fact 2} and Lemma~\ref{le:resolvents of sums of kernels}, this shows that Corollary~\ref{co:resolvent inequality} yields the results in \cite[Theorems~1.4 and~1.5]{Lin13}, \cite[Theorem~3.2]{ZhanWei16}, \cite[Theorems~8 and~9]{Alm17}, \cite[Theorem~3]{VanCap19} and \cite[Theorem~1]{LiuXuZhou24} under weaker conditions.
\end{Example}

Let us turn to powers of fractional kernels. That is, we directly analyse $k^{p}$ for $p\geq 1$ when $k$ is a fractional kernel. To this end, we introduce the entire function $\mathrm{E}_{\alpha,\beta,p}:\mathbb{C}\rightarrow\mathbb{C}$,
\begin{equation}\label{eq:extension of the Mittag-Leffler function}
\mathrm{E}_{\alpha,\beta,p}(z) := \sum_{n=0}^{\infty} \frac{z^{n}}{\Gamma(\alpha n + \beta)^{\frac{1}{p}}}
\end{equation}
for $\alpha > 0$ and $\beta\geq 0$ that agrees with the Mittag-Leffler function $\mathrm{E}_{\alpha,\beta}$ if $\beta > 0$ and $p = 1$. The fact that the power series representing $\mathrm{E}_{\alpha,\beta,p}$ has an infinite radius of convergence follows, as before, from the Cauchy-Hadamard Theorem and~\eqref{eq:limit related to the digamma function}.

\begin{Example}[Fractional kernels]\label{ex:fractional kernel}
Let $I$ be bounded from below. For $\alpha > 0$ and $\beta \geq 0$ with $\beta + 1 - \frac{1}{p} < \alpha$ we set $\alpha_{p} := (\alpha - 1)p + 1$ and suppose that
\begin{equation*}
k(t,s) = (t-s)^{\alpha - 1}(s - t_{0})^{-\beta}
\end{equation*}
for any $s,t\in I$ with $t_{0} < s < t$. Then Proposition~\ref{pr:family of functions} provides a sequence $(f_{p,n})_{n\in\N}$ of $]0,\infty[$-valued continuous functions on $]0,\infty[\times ]0,\infty[$ that are decreasing in the second variable such that
\begin{equation}\label{eq:fractional kernel 1}
\R_{k^{p},n}(t,s) = f_{p,n}(t-s,s-t_{0})
\end{equation}
for all $n\in\N$ and $s,t\in I$ with $t_{0} < s < t$. Namely, based on the inequality $\beta p < \alpha_{p}$, we have $f_{p,1}(x,y) = x^{\alpha_{p} - 1}y^{-\beta p}$ for all $x, y > 0$ and the recursion
\begin{equation*}
f_{p,n+1}(x,y) = x^{\alpha_{p}}\int_{0}^{1}\!(1 - \lambda)^{\alpha_{p} - 1}(\lambda x + y)^{-\beta p}f_{p,n}(\lambda x,y)\,\mathrm{d}\lambda
\end{equation*}
holds for any $n\in\N$ and $x, y > 0$. Further, we set $\hat{c}_{p,n} := \prod_{i=1}^{n-1}\frac{\Gamma((\alpha_{p} - \beta p)i)}{\Gamma((\alpha_{p} - \beta p)i + \beta p)}$ and readily note that $\hat{c}_{p,n} = 1$ if $\beta = 0$. Then Proposition~\ref{pr:family of functions} also implies the estimate
\begin{equation}\label{eq:fractional kernel 2}
f_{p,n}(x,y) \leq \frac{\hat{c}_{p,n}\Gamma(\alpha_{p})^{n}}{\Gamma((\alpha_{p} - \beta p)n + \beta p)}x^{(\alpha_{p} - \beta p)n + \beta p - 1}y^{-\beta p}
\end{equation}
for all $x,y > 0$, which turns into an equality if $\beta = 0$, and we have $\max_{n\in\N}\hat{c}_{p,n} = \hat{c}_{p}$ for $n_{p,\Gamma} := \min\{i\in\N\,|\, (\alpha_{p} - \beta p)i \geq x_{\Gamma}\}$ and $\hat{c}_{p} := \max_{i=1,\dots,n_{p,\Gamma}} \hat{c}_{p,i}$.

Since the kernel $k^{p}$ falls into the setting of Example~\ref{ex:sums of transformed fractional kernels}, Lemma~\ref{le:sufficient criterion} shows that for each measurable function $u_{0}:I\rightarrow [0,\infty]$ the condition~\eqref{eq:resolvent inequality condition 1} is valid if
\begin{equation}\label{eq:fractional kernel 3}
\int_{t_{0}}^{t}\!(t - t_{0})^{-\beta p}u_{0}(s)^{p}\,\mathrm{d}s < \infty\quad\text{for all $t\in I$.}
\end{equation}
Clearly, if $t_{0}\in I$, then~\eqref{eq:fractional kernel 3} is equivalent to the local $p$-fold integrability of the function $I\rightarrow [0,\infty]$, $t\mapsto (t-t_{0})^{-\beta}u_{0}(t)$. The sufficiency of~\eqref{eq:fractional kernel 3} is justified by the estimate~\eqref{eq:fractional kernel 2}, which implies that
\begin{equation*}
\R_{k^{p},n}(t,s) \leq \frac{\hat{c}_{p}\Gamma(\alpha_{p})^{n}}{\Gamma((\alpha_{p} - \beta p)n + \beta p)}(t-s)^{(\alpha_{p} - \beta p)n + \beta p -1}(s - t_{0})^{-\beta p}
\end{equation*}
for all $n\in\N$ and $s,t\in I$ with $t_{0} < s < t$. Finally, a substitution shows that the series $\mathrm{I}_{k,p}(t)$, introduced in~\eqref{eq:specific notation}, is of the form
\begin{equation*}
\mathrm{I}_{k,p}(t) = \sum_{n=1}^{\infty}\bigg(\int_{0}^{1}\!f_{p,n}\big((1-\lambda)(t - t_{0}),\lambda(t - t_{0})\big)\,\mathrm{d}\lambda\bigg)^{\frac{1}{p}}(t - t_{0})^{\frac{1}{p}}.
\end{equation*}
Thus, the estimate~\eqref{eq:fractional kernel 2} entails that $\mathrm{I}_{k,p}(t) < \infty$ $\Leftrightarrow$ $\beta p < 1$. In this case, the extension $\mathrm{E}_{\alpha_{p} - \beta p,1,p}$ of the Mittag-Leffler function satisfies
\begin{equation*}
\mathrm{I}_{k,p}(t) \leq \hat{c}_{p}^{\frac{1}{p}}\Gamma(1 - \beta p)^{\frac{1}{p}}\big(\mathrm{E}_{\alpha_{p} - \beta p,1,p}\big(\Gamma(\alpha_{p})^{\frac{1}{p}}(t-t_{0})^{\frac{\alpha_{p}}{p} - \beta}\big) - 1\big),
\end{equation*}
which is an identity if $\beta = 0$. By Example~\ref{ex:family of functions}, this analysis applies for $p=1$ to the kernels appearing in the inequalities in \cite[pp.~188-189]{Hen81}, \cite[Theorem~1]{YeGaoDin07} and \cite[Theorem~3.2]{Web19}.
\end{Example}

\subsection{Kernels on Cartesian products and Gronwall inequalities}\label{se:2.3}

By estimating the resolvent sequences of kernels on preordered Cartesian products that can be bounded by products of kernels on one preordered set, we deduce Gronwall inequalities for functions of several variables.

To this end, let $\mu$ be a $\sigma$-finite measure on $\mathcal{I}$, $k$ be a non-negative kernel on $I$ and $p\geq 1$. Then the hypothesis that $I$ is merely preordered allows for a product representation.

\begin{Proposition}\label{pr:kernels on Cartesian products}
For $m\in\N$ and any $i\in\{1,\dots,m\}$ let $I_{i}$ be a non-empty set endowed with a $\sigma$-field $\mathcal{I}_{i}$ and a preorder $\leq_{i}$ such that
\begin{equation*}
\Delta_{i}:= \{(t_{i},s_{i})\in I_{i}\times I_{i}\,|\, s_{i}\leq_{i} t_{i}\}\quad\text{belongs to}\quad \mathcal{I}_{i}\otimes\mathcal{I}_{i}.
\end{equation*}
In the setting $I = I_{1}\times\cdots\times I_{m}$, $\mathcal{I} = \mathcal{I}_{1}\otimes\cdots\otimes\mathcal{I}_{m}$, and $s\leq t$ $\Leftrightarrow$ $s_{1}\leq_{1} t_{1},\dots,s_{m}\leq_{m} t_{m}$ for all $s,t\in I$, the following two assertions hold:
\begin{enumerate}[(i)]
\item The triangular set of all $(t,s)\in I\times I$ with $s\leq t$ lies in $\mathcal{I}\otimes\mathcal{I}$ as preimage of $\Delta_{1}\times\cdots\times\Delta_{m}$ under a product measurable map on $I\times I$. Further,
\begin{equation*}
I(t) = I_{1}(t_{1})\times\cdots\times I_{m}(t_{m})\quad\text{and}\quad [r,t] = [r_{1},t_{1}]\times\cdots\times [r_{m},t_{m}]
\end{equation*}
for all $r,t\in I$ with $r\leq t$.

\item If $\mu_{i}$ is a $\sigma$-finite measure on $\mathcal{I}_{i}$ and $k_{i}$ is a non-negative kernel on $I_{i}$ for each $i\in\{1,\dots,m\}$ such that
\begin{equation}\label{eq:kernels on Cartesian products 1}
\mu \leq \mu_{1}\otimes\cdots\otimes\mu_{m}\quad\text{and}\quad k(t,s) \leq k_{1}(t_{1},s_{1})\cdots k_{m}(t_{m},s_{m})
\end{equation}
for any $s,t\in I$ with $s\leq t$, then
\begin{equation}\label{eq:kernels on Cartesian products 2}
\R_{k,\mu,n}(t,s) \leq \R_{k_{1},\mu_{1},n}(t_{1},s_{1})\cdots \R_{k_{m},\mu_{m},n}(t_{m},s_{m})
\end{equation}
for all $n\in\N$ and $s,t\in I$ with $s\leq t$. Moreover,
\begin{align}\label{eq:kernels on Cartesian products 3}
\mathrm{I}_{k,\mu,p}(t) &\leq \sum_{n=1}^{\infty}\prod_{i=1}^{m}\bigg(\int_{I_{i}(t_{i})}\!\R_{k_{i}^{p},\mu_{i},n}(t_{i},s_{i})\,\mu_{i}(\mathrm{d}s_{i})\bigg)^{\frac{1}{p}}\leq \prod_{i=1}^{m} \mathrm{I}_{k_{i},\mu_{i},p}(t_{i})
\end{align}
for every $t\in I$, and if the two inequalities in~\eqref{eq:kernels on Cartesian products 1} are equalities, then~\eqref{eq:kernels on Cartesian products 2} and the first inequality in~\eqref{eq:kernels on Cartesian products 3} turn into identities.
\end{enumerate}
\end{Proposition}

Let us focus on two specific settings of Proposition~\ref{pr:kernels on Cartesian products} that are based on Definition~\ref{de:set of kernels} and Examples~\ref{ex:regular kernels on intervals} and~\ref{ex:void preorder}:
\begin{enumerate}[label=(C.\arabic*), ref=C.\arabic*, leftmargin=\widthof{(C.3)} + \labelsep]
\setcounter{enumi}{2}
\item\label{co:3} There are $m\in\N$, non-degenerate intervals $I_{1},\dots,I_{m}$ in $\Re$ and a measurable space $(I_{m+1},\mathcal{I}_{m+1})$ such that $I = I_{1}\times\cdots\times I_{m+1}$, $ \mathcal{I} = \mathcal{B}(I_{1})\otimes\cdots\otimes\mathcal{B}(I_{m})\otimes\mathcal{I}_{m+1}$ and
\begin{equation*}
s\leq t\quad\Leftrightarrow\quad s_{1}\leq t_{1},\dots,s_{m}\leq t_{m}\quad\text{for any $s,t\in I$.}
\end{equation*}
For each $i\in\{1,\dots,m\}$ there is a $\sigma$-finite Borel measure $\mu_{i}$ on $I_{i}$ satisfying $\mu_{i}(\{t_{i}\})$ $= 0$ for all $t_{i}\in I_{i}$ and a $\sigma$-finite measure $\mu_{m+1}$ on $(I_{m+1},\mathcal{I}_{m+1})$ such that
\begin{equation*}
\mu = \mu_{1}\otimes\cdots\otimes\mu_{m+1}.
\end{equation*}
Moreover, there are $k_{1}\in K_{\mu_{1}}^{p}(I_{1}),\dots,k_{m}\in K_{\mu_{m}}^{p}(I_{m})$ and a measurable $p$-fold integrable function $k_{m+1}:I_{m+1}\rightarrow [0,\infty]$ such that
\begin{equation*}
k(t,s) = k_{1}(t_{1},s_{1})\cdots k_{m}(t_{m},s_{m})k_{m+1}(s_{m+1})\quad\text{for all $s,t\in I$ with $s\leq t$.}
\end{equation*}

\item\label{co:4} The preorder $\leq$ on $I$ is void, that is, $s\leq t$ holds for any $s,t\in I$, and there is a measurable function $k_{1}:I\rightarrow [0,\infty]$ satisfying
\begin{equation*}
k(\cdot,s) = k_{1}(s)\quad\text{for all $s\in I$}\quad\text{and}\quad\int_{I}\!k_{1}(t)^{p}\,\mu(\mathrm{d}t) < 1.
\end{equation*}
\end{enumerate}

\begin{Remark}
The setting~\eqref{co:3} includes the possibility that $I_{m+1}$ is a singleton, $\mu_{m+1}$ is a probability measure and $k_{m+1}$ equals one. That is, there is $t_{m+1}\in I_{m+1}$ such that
\begin{equation*}
I_{m+1} = \{t_{m+1}\},\quad \mathcal{I}_{m+1}=\{\emptyset,\{t_{m+1}\}\},\quad \mu_{m+1}(\{t_{m+1}\}) = 1\quad\text{and}\quad k_{m+1}(t_{m+1}) = 1.
\end{equation*}
In such a case, we may identify $I$ with $I_{1}\times\cdots\times I_{m}$ via the bijective map $I\rightarrow I_{1}\times\cdots\times I_{m}$, $s\mapsto (s_{1},\dots,s_{m})$.
\end{Remark}

By convention, we set $m:= 0$ whenever~\eqref{co:4} holds, which allows us to use the variable $m$ in either scenario.

\begin{Proposition}\label{pr:estimates for Gronwall inequalities}
Suppose that~\eqref{co:3} or~\eqref{co:4} is satisfied. Then
\begin{equation}\label{eq:estimates for Gronwall inequalities 1}
\R_{k^{p},\mu,n}(t,s) \leq \frac{k(t,s)^{p}}{((n-1)!)^{m}}\bigg(\int_{[s,t]}\!k(t,\tilde{s})^{p}\,\mu(\mathrm{d}\tilde{s})\bigg)^{n-1}
\end{equation}
for any $n\in\N$ and $s,t\in I$ with $s\leq t$, and equality holds if $m = 0$ or both $m\geq 1$ and for each $i\in\{1,\dots,m\}$ there is a measurable function $\tilde{k}_{i}:I_{i}\rightarrow [0,\infty]$ such that
\begin{equation}\label{eq:specific case}
k_{i}(t_{i},s_{i}) = \tilde{k}_{i}(s_{i})\quad\text{and}\quad\int_{I_{i}(t_{i})}\!\tilde{k}_{i}(\tilde{s}_{i})^{p}\,\mu_{i}(\mathrm{d}\tilde{s}_{i}) < \infty
\end{equation}
for all $s_{i},t_{i}\in I_{i}$ with $s_{i}\leq t_{i}$. Moreover, for a measurable function $u_{0}:I\rightarrow [0,\infty]$ the condition~\eqref{eq:resolvent inequality condition 2} is valid if and only if
\begin{equation}\label{eq:estimates for Gronwall inequalities 2}
\int_{I(t)}\!k(t,s)^{p}u_{0}(s)^{p}\,\mu(\mathrm{d}s) < \infty
\end{equation}
for any $t\in I$. In particular, $k\in K_{\mu}^{p}(I)$, and~\eqref{eq:estimates for Gronwall inequalities 2} holds if $u_{0}$ is essentially bounded on $I(t)$.
\end{Proposition}

For a measurable function $v_{0}:I\rightarrow [0,\infty]$ and a non-negative kernel $l$ on $I$, we define another measurable function $v:I\rightarrow [0,\infty]$ by 
\begin{equation*}
v(t) := v_{0}(t) + \bigg(\int_{I(t)}\!l(t,s)^{p}\,\mu(\mathrm{d}s)\bigg)^{\frac{1}{p}}.
\end{equation*}
Then, under~\eqref{co:3} or~\eqref{co:4}, Proposition~\ref{pr:estimates for Gronwall inequalities} yields the estimate
\begin{equation}\label{eq:estimates for Gronwall inequalities 3}
\int_{I(t)}\!\R_{k^{p},\mu,n}(t,s)v(s)^{p}\,\mu(\mathrm{d}s)\leq\int_{I(t)}\!\frac{k(t,s)^{p}}{((n-1)!)^{m}}\bigg(\int_{[s,t]}\!k(t,\tilde{s})^{p}\,\mu(\mathrm{d}\tilde{s})\bigg)^{n-1}v(s)^{p}\,\mu(\mathrm{d}s)
\end{equation}
for all $n\in\N$ and $t\in I$, which turns into an equation if $m = 0$ or the special setting in~\eqref{eq:specific case} occurs for $m\geq 1$. To bound the integral on the right-hand side, we impose the subsequent conditions on the kernel $l$ in each of the two scenarios:
\begin{enumerate}[label=(C.\arabic*), ref=C.\arabic*, leftmargin=\widthof{(C.5)} + \labelsep]
\setcounter{enumi}{4}
\item\label{co:5} In addition to~\eqref{co:3} there are $l_{1}\in K_{\mu_{1}}^{p}(I_{1}),\dots,l_{m}\in K_{\mu_{m}}^{p}(I_{m})$ and a measurable $p$-fold integrable function  $l_{m+1}:I_{m+1}\rightarrow [0,\infty]$ such that
\begin{equation*}
l(t,s) = l_{1}(t_{1},s_{1})\cdots l_{m}(t_{m},s_{m})l_{m+1}(s_{m+1})\quad\text{for any $s,t\in I$ with $s\leq t$.}
\end{equation*}

\item\label{co:6} The scenario~\eqref{co:4} is valid and there is a measurable $p$-fold integrable function $l_{1}:I\rightarrow [0,\infty]$ such that $l(\cdot,s) = l_{1}(s)$ for every $s\in I$.
\end{enumerate}

Now we can continue the estimation~\eqref{eq:estimates for Gronwall inequalities 3} as follows.

\begin{Lemma}\label{le:integral estimates}
Under~\eqref{co:5} or~\eqref{co:6}, we have
\begin{align*}
\bigg(\int_{I(t)}\!\frac{k(t,s)^{p}}{((n-1)!)^{m}}\bigg(\int_{[s,t]}&\!k(t,\tilde{s})^{p}\,\mu(\mathrm{d}\tilde{s})\bigg)^{n-1}v(s)^{p}\,\mu(\mathrm{d}s)\bigg)^{\frac{1}{p}}\\
&\leq \sup_{s\in I(t)} v_{0}(s)\frac{1}{(n!)^{\frac{m}{p}}}\bigg(\int_{I(t)}\!k(t,\tilde{s})^{p}\,\mu(\mathrm{d}\tilde{s})\bigg)^{\frac{n}{p}}\\
&\quad + \frac{1}{(n!)^{\frac{m}{p}}}\bigg(\int_{I(t)}\!\bigg(\int_{[s,t]}\!k(t,\tilde{s})^{p}\,\mu(\mathrm{d}\tilde{s})\bigg)^{n}l(t,s)^{p}\,\mu(\mathrm{d}s)\bigg)^{\frac{1}{p}}
\end{align*}
for any $n\in\N$ and $t\in I$, and equality holds in the following two cases:
\begin{enumerate}[(i)]
\item $v_{0}$ is constant and $l$ vanishes.

\item $v_{0} = 0$ and if $m\geq 1$, then for every $i\in\{1,\dots,m\}$ there is a measurable function $\tilde{l}_{i}:I_{i}\rightarrow [0,\infty]$ such that $l_{i}(t_{i},s_{i}) = \tilde{l}_{i}(s_{i})$ and $\int_{I_{i}(t_{i})}\!\tilde{l}_{i}(\tilde{s}_{i})^{p}\,\mu_{i}(\mathrm{d}\tilde{s}_{i}) < \infty$ for any $s_{i},t_{i}\in I_{i}$ with $s_{i}\leq t_{i}$.
\end{enumerate}
\end{Lemma}

Consequently, Proposition~\ref{pr:resolvent sequence inequality} entails two \emph{$L^{p}$-Gronwall sequence inequalities} in different versions of sharpness.

\begin{Corollary}\label{co:Gronwall sequence inequalities}
Let~\eqref{co:5} or~\eqref{co:6} hold and $J$ be a set in $I$ of full measure. If $(u_{n})_{n\in\N_{0}}$ is a sequence of $[0,\infty]$-valued measurable functions on $I$ satisfying~\eqref{eq:resolvent sequence inequality 1} for all $n\in\N$ and $t\in J$, then
\begin{equation}\label{eq:Gronwall sequence inequalities}
\begin{split}
u_{n}(t) &\leq v(t)  + w_{n}(t) + \sum_{i=0}^{n-2}\bigg(\int_{I(t)}\!\frac{k(t,s)^{p}}{(i!)^{m}}\bigg(\int_{[s,t]}\!k(t,\tilde{s})^{p}\,\mu(\mathrm{d}\tilde{s})\bigg)^{i}v(s)^{p}\,\mu(\mathrm{d}s)\bigg)^{\frac{1}{p}}\\
&\leq \sup_{s\in I(t)} v_{0}(s)\sum_{i=0}^{n-1}\frac{1}{(i!)^{\frac{m}{p}}}\bigg(\int_{I(t)}\!k(t,\tilde{s})^{p}\,\mu(\mathrm{d}\tilde{s})\bigg)^{\frac{i}{p}}\\
&\quad + w_{n}(t) + \sum_{i=0}^{n-1}\frac{1}{(i!)^{\frac{m}{p}}}\bigg(\int_{I(t)}\!\bigg(\int_{[s,t]}\!k(t,\tilde{s})^{p}\,\mu(\mathrm{d}\tilde{s})\bigg)^{i}l(t,s)^{p}\,\mu(\mathrm{d}s)\bigg)^{\frac{1}{p}}
\end{split}
\end{equation}
for any $n\in\N$ and $t\in J$, where the measurable function $w_{n}:I\rightarrow [0,\infty]$ is given by
\begin{equation*}
w_{n}(t) := \bigg(\int_{I(t)}\!\frac{k(t,s)^{p}}{((n-1)!)^{m}}\bigg(\int_{[s,t]}\!k(t,\tilde{s})^{p}\,\mu(\mathrm{d}\tilde{s})\bigg)^{n-1}u_{0}(s)^{p}\,\mu(\mathrm{d}s)\bigg)^{\frac{1}{p}}.
\end{equation*}
Moreover, the first inequality in~\eqref{eq:Gronwall sequence inequalities} turns into an identity if $p = 1$,~\eqref{eq:resolvent sequence inequality 1} is an equation and either $m = 0$ or the specific case in~\eqref{eq:specific case} holds for $m\geq 1$.
\end{Corollary}

As a result, from either Corollary~\ref{co:resolvent inequality} or Corollary~\ref{co:Gronwall sequence inequalities} we directly deduce two types of \emph{$L^{p}$-Gronwall inequalities}.

\begin{Corollary}\label{co:Gronwall inequalities}
Let~\eqref{co:5} or~\eqref{co:6} be valid and $J$ be a set in $I$ of full measure. If $u,u_{0}:I\rightarrow [0,\infty]$ are measurable and satisfy $u_{0}(t)\leq u(t)$,~\eqref{eq:resolvent inequality 1} and
\begin{equation*}
\int_{I(t)}\!k(t,s)^{p}u_{0}(s)^{p}\,\mu(\mathrm{d}s) < \infty
\end{equation*}
for every $t\in J$, then
\begin{equation}\label{eq:Gronwall inequalities}
\begin{split}
u(t) &\leq v(t) + \sum_{n=0}^{\infty}\frac{1}{(n!)^{\frac{m}{p}}}\bigg(\int_{I(t)}\!k(t,s)^{p}\bigg(\int_{[s,t]}\!k(t,\tilde{s})^{p}\,\mu(\mathrm{d}\tilde{s})\bigg)^{n}v(s)^{p}\,\mu(\mathrm{d}s)\bigg)^{\frac{1}{p}}\\
&\leq \sup_{s\in I(t)} v_{0}(s)\sum_{n=0}^{\infty}\frac{1}{(n!)^{\frac{m}{p}}}\bigg(\int_{I(t)}\!k(t,\tilde{s})^{p}\,\mu(\mathrm{d}\tilde{s})\bigg)^{\frac{n}{p}}\\
&\quad + \sum_{n=0}^{\infty}\frac{1}{(n!)^{\frac{m}{p}}}\bigg(\int_{I(t)}\!\bigg(\int_{[s,t]}\!k(t,\tilde{s})^{p}\,\mu(\mathrm{d}\tilde{s})\bigg)^{n}l(t,s)^{p}\,\mu(\mathrm{d}s)\bigg)^{\frac{1}{p}}
\end{split}
\end{equation}
for each $t\in J$ and the last two series are finite. Further, the first estimate in~\eqref{eq:Gronwall inequalities} is an equation if $p = 1$,~\eqref{eq:resolvent inequality 1} is an identity and either $m = 0$ or~\eqref{eq:specific case} holds for $m\geq 1$.
\end{Corollary}

\begin{Example}
For $m = p = 1$ we obtain two \emph{Gronwall inequalities in time and space variables}. Namely, the estimates~\eqref{eq:Gronwall inequalities} reduce to
\begin{align*}
u(t) &\leq v(t) + \int_{I(t)}\!k(t,s)e^{\int_{[s,t]}\!k(t,\tilde{s})\,\mu(\mathrm{d}\tilde{s})}v(s)\,\mu(\mathrm{d}s)\\
&\leq e^{\int_{I(t)}\!k(t,s)\,\mu(\mathrm{d}s)}\sup_{s'\in I(t)}v_{0}(s') + \int_{I(t)}\!e^{\int_{[s,t]}\!k(t,\tilde{s})\,\mu(\mathrm{d}\tilde{s})}l(t,s)\,\mu(\mathrm{d}s)
\end{align*}
for any $t\in J$. If instead $m = 0$, then we recover \emph{Gronwall inequalities of Fredholm type} in the $L^{p}$-norm, since in this case the estimates~\eqref{eq:Gronwall inequalities} become
\begin{align*}
u(t) &\leq v(t) + \bigg(1 - \bigg(\int_{I}\!k_{1}(s)^{p}\,\mu(\mathrm{d}s)\bigg)^{\frac{1}{p}}\bigg)^{-1}\bigg(\int_{I}\!k_{1}(s)^{p}v(s)^{p}\,\mu(\mathrm{d}s)\bigg)^{\frac{1}{p}}\\
&\leq \bigg(1 - \bigg(\int_{I}\!k_{1}(s)^{p}\,\mu(\mathrm{d}s)\bigg)^{\frac{1}{p}}\bigg)^{-1}\bigg(\sup_{s'\in I}v_{0}(s') + \bigg(\int_{I}\!l_{1}(s)^{p}\,\mu(\mathrm{d}s)\bigg)^{\frac{1}{p}}\bigg)\quad\text{for all $t\in J$.}
\end{align*}
\end{Example}

\subsection{General applications of the fixed point result}\label{se:2.4}

We discuss two kinds of applications of Theorem~\ref{th:fixed point} in which $X$ is a metric space or $(d_{t})_{t\in I}$ is an increasing family of pseudometrics. In this connection, sufficient conditions for the uniqueness of fixed points and the sequential continuity of $\Psi$ are given.

Let us recall that initially $(d_{t})_{t\in I}$ is a family of $[0,\infty]$-valued functions on $X\times X$ satisfying~\eqref{eq:pseudometrical condition 1} and~\eqref{eq:pseudometrical condition 2} for all $x,\tilde{x}\in X$. Clearly, there could be just one premetric $d$ on $X$ such that
\begin{equation*}
d_{t} = d\quad\text{for all $t\in I$,}
\end{equation*}
in which case the required completeness of $(d_{t})_{t\in I}$ in Theorem~\ref{th:fixed point} is equivalent to that of $d$. In particular, this includes spaces of maps endowed with the topology of uniform convergence.

\begin{Example}[Supremum premetric]\label{ex:supremum premetric}
Suppose that $E$ is a premetrisable space and $\rho$ is a premetric inducing its topology. On the set $E^{I}$ of all $E$-valued maps on $I$ we define a premetric $d_{\infty}$ by 
\begin{equation*}
d_{\infty}(x,\tilde{x}) := \sup_{t\in I}\rho\big(x(t),\tilde{x}(t)\big).
\end{equation*}
Then a sequence in $E^{I}$ converges with respect to $d_{\infty}$ if and only if it converges uniformly, and $d_{\infty}$ is complete if $\rho$ satisfies this property. Further, let us call a map $x:I\rightarrow E$ bounded if
\begin{equation*}
\sup_{t\in I} \rho(x(t),u) < \infty\quad\text{for some $u\in E$,}
\end{equation*}
and note that the set of all $E$-valued bounded maps on $I$ is closed relative to $d_{\infty}$. So, let $X\subseteq E^{I}$ and $d$ be the restriction of $d_{\infty}$ to $X\times X$. Then the following two statements hold:
\begin{enumerate}[(1)]
\item If $\rho$ is complete, then $X$ is closed with respect to $d_{\infty}$ if and only if $d$ is complete.

\item If $\rho$ is a metric and $X$ consists of bounded maps only, then $d$ is a metric.
\end{enumerate}
\end{Example}

By extending Example~\ref{ex:supremum premetric}, we obtain an increasing family $(d_{t})_{t\in I}$ of pseudopremetrics and allow for spaces of maps equipped with the topology of local uniform convergence.

\begin{Example}\label{ex:increasing family of supremum pseudopremetrics}
Let $E$ be a premetrisable space and $\rho$ be a premetric inducing its topology. We endow $E^{I}$ with the topology of converge with respect to the pseudopremetric $d_{\infty,t}$ on $E^{I}$ given by
\begin{equation*}
d_{\infty,t}(x,\tilde{x}) := \sup_{s\in I(t)} \rho\big(x(s),\tilde{x}(s)\big)\quad\text{for any $t\in I$.}
\end{equation*}
That is, a sequence in $E^{I}$ converges to a map $x:I\rightarrow E$ if and only if it converges uniformly to $x$ on $I(t)$ for each $t\in I$. Then $(d_{\infty,t})_{t\in I}$ is complete as soon as $\rho$ is, and the set
\begin{equation*}
\{x:I\rightarrow E\,|\,\text{$x$ is bounded on $I(t)$ for any $t\in I$}\}
\end{equation*}
is closed. Thus, we let $X\subseteq E^{I}$ and $d_{t}$ be the restriction of $d_{\infty,t}$ to $X\times X$ for each $t\in I$. Then the subsequent two statements are valid:
\begin{enumerate}[(1)]
\item If $\rho$ is complete, then $X$ is closed in $E^{I}$ if and only if $(d_{t})_{t\in I}$ is complete.

\item If $\rho$ is a metric and each $x\in X$ is bounded on $I(t)$ for any $t\in I$, then $d_{t}$ is a pseudometric for each $t\in I$.
\end{enumerate}
In particular, if $I$ is a locally compact metrisable space, $I(t)$ is relatively compact for each $t\in I$ and for every compact set $K$ in $I$ we have
\begin{equation*}
K\subseteq I(t)\quad\text{for some $t\in I$,}
\end{equation*}
then a sequence $(x_{n})_{n\in\N}$ in $X$ converges to some $x\in X$ in the sense of~\eqref{eq:underlying topology} if and only if $(x_{n})_{n\in\N}$ converges locally uniformly to $x$.
\end{Example}

Next, let us simplify the assumptions of Lemma~\ref{le:uniqueness of fixed points}, which leads to unique fixed points.

\begin{Lemma}\label{le:uniqueness of fixed points 2}
Under~\eqref{co:1}, each of the following conditions entails the subsequent one:
\begin{enumerate}[(i)]
\item The family $(d_{t})_{t\in I}$ is increasing and $\mathrm{I}_{\lambda,\mu,p}(t) < \infty$ for all $t\in I$.

\item We have $\sum_{n=1}^{\infty}(\int_{I(t)}\!\R_{\lambda^{p},\mu,n}(t,s)\Lambda(s,x,\tilde{x})^{p}\,\mu(\mathrm{d}s))^{\frac{1}{p}} < \infty$ for any $t\in I$ and $x,\tilde{x}\in X$ with $d_{t}(x,\tilde{x}) < \infty$.

\item The limit~\eqref{eq:condition for unique fixed points} and the hypothesis~\eqref{eq:fixed point condition} hold for all $t\in I$ and $x,\tilde{x},x_{0}\in X$ for which $d_{t}(x,\tilde{x})$ and $d_{t}(x_{0},\Psi(x_{0}))$ are finite, respectively.
\end{enumerate}
In particular, if $d_{t} < \infty$ for each $t\in I$ and one of these conditions holds, then there is at most a unique fixed point of $\Psi$.
\end{Lemma}

The assumed sequentially continuity of $\Psi$ in Theorem~\ref{th:fixed point} can be ensured as follows.

\begin{Lemma}\label{le:sequential continuity}
Let~\eqref{co:1} and~\eqref{co:2} be valid. Then each of the following hypotheses implies the subsequent one:
\begin{enumerate}[(i)]
\item The family $(d_{t})_{t\in I}$ is increasing and $\int_{I(t)}\!\lambda(t,s)^{p}\,\mu(\mathrm{d}s) < \infty$ for all $t\in I$.

\item For each $t\in I$ and every sequence $(x_{n})_{n\in\N}$ in $X$ that converges to some $x\in X$, there is $n_{0}\in\N$ such that $(\lambda(t,\cdot)^{p}\Lambda(\cdot,x_{n},x)^{p})_{n\in\N:\,n\geq n_{0}}$ is uniformly integrable.

\item The map $\Psi$ is sequentially continuous.
\end{enumerate}
\end{Lemma}

To motivate applications of Theorem~\ref{th:fixed point} when $(I,\mathcal{I})$ is merely a measurable space and $X$ is a metric space, we recall Example~\ref{ex:supremum premetric} and generalise Example~\ref{ex:fixed point thm} as follows.

\begin{Example}
Let $\leq$ be void and $d_{t} = d$ for all $t\in I$ and a metric $d$ inducing the topology of $X$. Then~\eqref{co:1} follows from the subsequent regularity condition on $\Psi$:
\begin{enumerate}[(1)]
\item There are $\Lambda:I\times X\times X\rightarrow [0,\infty]$, a $\sigma$-finite measure $\mu$ on $\mathcal{I}$, $p\geq 1$ and a measurable function $\lambda:I\rightarrow [0,\infty]$ such that $\Lambda(\cdot,x,\tilde{x})$ is measurable,
\begin{equation*}
d\big(\Psi(x),\Psi(\tilde{x})\big) \leq \bigg(\int_{I}\!\lambda(t)^{p}\Lambda(t,x,\tilde{x})^{p}\,\mu(\mathrm{d}t)\bigg)^{\frac{1}{p}}
\end{equation*}
and $\Lambda(\cdot,x,\tilde{x}) \leq d(x,\tilde{x})$ for all $x,\tilde{x}\in X$.
\end{enumerate}

Hence, let condition~(1) hold and $\lambda_{0} :=(\int_{I}\!\lambda(t)^{p}\,\mu(\mathrm{d}t))^{\frac{1}{p}}$ be finite. Then $\lambda_{0}$ is a Lipschitz constant of $\Psi$ and from Example~\ref{ex:void preorder} and Lemma~\ref{le:uniqueness of fixed points 2} we infer that if
\begin{equation*}
\lambda_{0} < 1,
\end{equation*}
then~\eqref{eq:condition for unique fixed points} and~\eqref{eq:fixed point condition} are satisfied for all $t\in I$ and $x,\tilde{x},x_{0}\in X$ and there is at most a unique fixed point of $\Psi$.

If in addition $d$ is complete, then $\Psi$ has a unique fixed point $\hat{x}$, and for any $x_{0}\in X$ it is the limit of the sequence $(x_{n})_{n\in\N}$ in $X$ recursively given by $x_{n}:= \Psi(x_{n-1})$, by Theorem~\ref{th:fixed point}. Moreover,~\eqref{eq:fixed point error estimate} becomes
\begin{equation*}
d(x_{n},\hat{x})\leq \bigg(\int_{I}\!\lambda(t)^{p}\Lambda(t,x_{0},\Psi(x_{0}))^{p}\,\mu(\mathrm{d}t)\bigg)^{\frac{1}{p}}\frac{\lambda_{0}^{n-1}}{1-\lambda_{0}}
\end{equation*}
for each $n\in\N$ and this bound cannot exceed the error estimate $d(x_{0},\Psi(x_{0}))\frac{\lambda_{0}^{n}}{1-\lambda_{0}}$ in Banach's fixed point theorem.
\end{Example}

Finally, based on Example~\ref{ex:increasing family of supremum pseudopremetrics}, let us indicate applications of Theorem~\ref{th:fixed point} when $d_{t}$ is a pseudometric but $\Psi$ may fail to be a contraction relative to $d_{t}$ for any $t\in I$.

\begin{Example}
Let $I$ be a non-degenerate interval in $\Re$, $\mathcal{I} =\mathcal{B}(I)$ and $\leq$ be the ordinary order. Then the following regularity condition on $\Psi$ implies~\eqref{co:1}:
\begin{enumerate}[(1)]
\item There are $\Lambda:I\times X\times X\rightarrow [0,\infty]$, a $\sigma$-finite Borel measure $\mu$ on $I$, $p\geq 1$ and a non-negative kernel $\lambda$ on $I$ satisfying~\eqref{eq:monotonicity condition} such that $\Lambda(\cdot,x,\tilde{x})$ is measurable,
\begin{equation*}
d_{t}\big(\Psi(x),\Psi(\tilde{x})\big) \leq \bigg(\int_{I(t)}\!\lambda(t,s)^{p}\Lambda(s,x,\tilde{x})^{p}\,\mu(\mathrm{d}s)\bigg)^{\frac{1}{p}},
\end{equation*}
$\Lambda(t,x,\tilde{x}) \leq d_{t}(x,\tilde{x})$ and $\mu(\{t\}) = 0$ for any $t\in I$ and $x,\tilde{x}\in X$.
\end{enumerate}

Thus, let us suppose that $(d_{t})_{t\in I}$ is an increasing family of pseudometrics, condition~(1) holds and the measurable function $\lambda_{0}:I\rightarrow [0,\infty]$ defined by
\begin{equation*}
\lambda_{0}(t):= \bigg(\int_{I(t)}\!\lambda(t,s)^{p}\,\mu(\mathrm{d}s)\bigg)^{\frac{1}{p}}
\end{equation*}
is finite. Then $\lambda_{0}(t)$ is a Lipschitz constant of $\Psi$ relative to $d_{t}$ and~\eqref{eq:condition for unique fixed points} and~\eqref{eq:fixed point condition} hold for all $t\in I$ and $x,\tilde{x},x_{0}\in X$, by Example~\ref{ex:regular kernels on intervals} and Lemma~\ref{le:uniqueness of fixed points 2}. In particular, there is at most a unique fixed point of $\Psi$.

Further, if $(d_{t})_{t\in I}$ is complete, then $\Psi$ has a unique fixed point $\hat{x}$, by Theorem~\ref{th:fixed point}. In this case, from the error estimate~\eqref{eq:fixed point error estimate} for the Picard sequence $(x_{n})_{n\in\N}$ we obtain that
\begin{align*}
d_{t}(x_{n},\hat{x}) &\leq \sum_{i=n}^{\infty}\bigg(\int_{I(t)}\!\frac{\lambda(t,s)^{p}}{(i-1)!}\bigg(\int_{s}^{t}\!\lambda(t,\tilde{s})^{p}\,\mu(\mathrm{d}\tilde{s})\bigg)^{i-1}\Lambda(s,x_{0},\Psi(x_{0}))^{p}\,\mu(\mathrm{d}s)\bigg)^{\frac{1}{p}}\\
&\leq d_{t}\big(x_{0},\Psi(x_{0})\big)\sum_{i=n}^{\infty}\bigg(\frac{1}{i!}\bigg)^{\frac{1}{p}}\lambda_{0}(t)^{i}
\end{align*}
for all $n\in\N$ and fixed $t\in I$, and the last term converges to zero as $n\uparrow\infty$, as the radius of convergence of the power series representing $\mathrm{E}_{1,1,p}$, which is generally given by~\eqref{eq:extension of the Mittag-Leffler function}, is infinite.

In particular, a fixed point approach of this type can be used to derive unique strong solutions to \emph{McKean-Vlasov SDEs}, as shown in \cite[Theorem~3.24]{KalMeyPro24} and \cite[Theorem~3.1]{KalMeyPro24-2}.
\end{Example}

\section{Proofs of the main and consequential results}\label{se:3}

\subsection{Proofs of the resolvent inequalities and the fixed point theorem}

\begin{proof}[Proof of Lemma~\ref{le:induction principle}]
The asserted estimate follows immediately by induction, since if $u_{n} \leq \Psi^{n}(u_{0})$ on $J$ for some $n\in\N$, then we have $u_{n+1} \leq \Psi(u_{n}) \leq \Psi\circ \Psi^{n}(u_{0}) = \Psi^{n+1}(u_{0})$ on $J$. Further, if the inequalities in~\eqref{eq:induction principle 1} and~\eqref{eq:induction principle 2} are equations, then the two inequalities in the preceding implication turn into equations as soon as we assume that $u_{n} = \Psi^{n}(u_{0})$ on $J$ for some $n\in\N$.
\end{proof}

\begin{proof}[Proof of Proposition~\ref{pr:resolvent sequence inequality}]
By Lemma~\ref{le:induction principle} and Example~\ref{ex:induction principle}, it suffices to the consider the case that equality holds in~\eqref{eq:resolvent sequence inequality 1}. That is, $u_{n} = \Psi^{n}(u_{0})$ for any $n\in\N$, where $\mathcal{D}$ is the convex cone of all $[0,\infty]$-valued measurable functions on $I$, the operator $\Psi:\mathcal{D}\rightarrow\mathcal{D}$ is given by
\begin{equation*}
\Psi(u)(t) := v(t) + \bigg(\int_{I(t)}\!k(t,s)^{p}u(s)^{p}\,\mu(\mathrm{d}s)\bigg)^{\frac{1}{p}}
\end{equation*}
and $\Psi^{n}$ is the $n$-fold composition of $\Psi$ with itself. So, let us show the claim by induction over $n\in\N$. Since for $n=1$ the claimed estimate already agrees with the imposed equation~\eqref{eq:resolvent sequence inequality 1}, we may assume that the assertion holds for some $n\in\N$. Then
\begin{equation}\label{eq:resolvent sequence auxiliary inequality}
\begin{split}
u_{n+1}(t) &\leq v(t) + \bigg(\int_{I(t)}\!k(t,s)^{p}v(s)^{p}\,\mu(\mathrm{d}s)\bigg)^{\frac{1}{p}}\\
&\quad + \sum_{i=1}^{n-1}\bigg(\int_{I(t)}\!\int_{[s,t]}\!k(t,\tilde{s})^{p}\R_{k^{p},\mu,i}(\tilde{s},s)\,\mu(\mathrm{d}\tilde{s})\,v(s)^{p}\,\mu(\mathrm{d}s)\bigg)^{\frac{1}{p}}\\
&\quad + \bigg(\int_{I(t)}\!\int_{[s,t]}\! k(t,\tilde{s})^{p}\R_{k^{p},\mu,n}(\tilde{s},s)\,\mu(\mathrm{d}\tilde{s})\,u_{0}(s)^{p}\,\mu(\mathrm{d}s)\bigg)^{\frac{1}{p}}
\end{split}
\end{equation}
for all $t\in J$, by induction hypothesis, Minkowski's inequality and Fubini's theorem. The desired estimate follows, since the general recursive description in~\eqref{eq:resolvent sequence} entails that
\begin{equation*}
\R_{k^{p},\mu,i+1}(t,s) = \int_{[s,t]}\!k(t,\tilde{s})^{p}\R_{k^{p},\mu,i}(\tilde{s},s)\,\mu(\mathrm{d}\tilde{s})
\end{equation*}
for any $i\in\{1,\dots,n\}$ and $s,t\in I$ with $s\leq t$. Further, for $p = 1$ we do not need to apply Minkowski's inequality and~\eqref{eq:resolvent sequence auxiliary inequality} turns into an identity once equality holds in~\eqref{eq:resolvent sequence inequality 2} for each $t\in J$.
\end{proof}

\begin{proof}[Proof of Corollary~\ref{co:resolvent inequality}]
The assertion follows from Proposition~\ref{pr:resolvent sequence inequality} by taking $u_{n} = u$ for all $n\in\N$, as the increasing sequence $(\sum_{i=1}^{n-1}(\int_{I(t)}\!\R_{k^{p},\mu,i}(t,s)v(s)^{p}\,\mu(\mathrm{d}s))^{\frac{1}{p}})_{n\in\N}$ converges to its supremum for any $t\in I$.
\end{proof}

\begin{proof}[Proof of Lemma~\ref{le:uniqueness of fixed points}]
Let $x$ and $\tilde{x}$ be two fixed points of $\Psi$. Then $u,u_{0}:I\rightarrow [0,\infty]$ given by $u(t) := d_{t}(x,\tilde{x})$ and $u_{0}(t) := \Lambda(t,x,\tilde{x})$ are measurable and satisfy $u_{0}(t) \leq u(t)$ and $u(t)^{p} \leq \int_{I(t)}\!\lambda(t,s)^{p}u_{0}(s)^{p}\,\mu(\mathrm{d}s)$ for all $t\in I$. Hence, Corollary~\ref{co:resolvent inequality} gives $u=0$, which in turn yields that $x = \tilde{x}$, by condition~\eqref{eq:pseudometrical condition 1}.
\end{proof}

\begin{proof}[Proof of Theorem~\ref{th:fixed point}]
We merely have to prove the existence and convergence assertions, as Lemma~\ref{le:uniqueness of fixed points} shows uniqueness. To this end, let us check that $(x_{n})_{n\in\N}$ is a Cauchy sequence. First,
\begin{equation*}
d_{t}(x_{n},x_{n+1}) \leq \bigg(\int_{I(t)}\!\R_{\lambda^{p},\mu,n}(t,s)\Lambda(s,x_{0},\Psi(x_{0}))^{p}\,\mu(\mathrm{d}s)\bigg)^{\frac{1}{p}}
\end{equation*}
for all $n\in\N$ and fixed $t\in I$, as an application of Proposition~\ref{pr:resolvent sequence inequality} shows. For this reason, the triangle inequality yields that
\begin{equation}\label{eq:Picard estimation}
d_{t}(x_{n},x_{m}) \leq \sum_{i=n}^{m-1}\bigg(\int_{I(t)}\!\R_{\lambda^{p},\mu,i}(t,s)\Lambda(s,x_{0},\Psi(x_{0}))^{p}\,\mu(\mathrm{d}s)\bigg)^{\frac{1}{p}}
\end{equation}
for any $m,n\in\N$ with $m > n$. Thus, $(x_{n})_{n\in\N}$ is a Cauchy sequence with respect to the pseudopremetric $d_{t}$. Namely, $\lim_{n\uparrow\infty}\sup_{m\in\N:\,m\geq n} d_{t}(x_{n},x_{m}) = 0$. Since $t\in I$ was arbitrarily chosen, there exists a unique $\hat{x}\in X$ to which $(x_{n})_{n\in\N}$ converges.

As a result, the error estimate~\eqref{eq:fixed point error estimate} follows from~\eqref{eq:Picard estimation} by taking the limit $m\uparrow\infty$. Finally, the sequential continuity of $\Psi$ gives $\lim_{n\uparrow\infty} x_{n+1} = \Psi(\hat{x})$, which in turn yields that $\hat{x} = \Psi(\hat{x})$.
\end{proof}

\begin{proof}[Proof of Lemma~\ref{le:uniqueness of fixed points 2}]
Since the second assertion is a direct consequence of Lemma~\ref{le:uniqueness of fixed points}, it suffices to verify the two claimed implications. (i) $\Rightarrow$ (ii): According to~\eqref{co:1}, we have $\sum_{n=1}^{\infty}(\int_{I(t)}\!\R_{\lambda^{p},\mu,n}(t,s)\Lambda(s,x,\tilde{x})^{p}\,\mu(\mathrm{d}s))^{\frac{1}{p}} $ $\leq d_{t}(x,\tilde{x})\mathrm{I}_{\lambda,\mu,p}(t)$ for any $t\in I$ and $x,\tilde{x}\in X$.

(ii) $\Rightarrow$ (iii): The condition~\eqref{eq:condition for unique fixed points} is necessary for the convergence of the series appearing in~(ii) for any $t\in I$ and $x,\tilde{x}\in X$. Moreover,~\eqref{eq:fixed point condition} is merely a special case of the finiteness condition in~(ii) when $x=x_{0}$ and $\tilde{x} = \Psi(x_{0})$ for each $x_{0}\in X$.
\end{proof}

\begin{proof}[Proof of Lemma~\ref{le:sequential continuity}]
(i) $\Rightarrow$ (ii): For any $t\in I$ and each sequence $(x_{n})_{n\in\N}$ in $X$ converging to some $x\in X$, there are $n_{0}\in\N$ and $c\geq 0$ such that $\Lambda(s,x_{n},x)\leq d_{t}(x_{n},x)\leq c$ for all $n\in\N$ with $n\geq n_{0}$ and $s\in I(t)$. Thus, as $(\lambda(t,\cdot)^{p}\Lambda(\cdot,x_{n},x)^{p})_{n\in\N:\,n\geq n_{0}}$ is dominated by an $[0,\infty]$-valued measurable integrable function, it is uniformly integrable.

(ii) $\Rightarrow$ (iii): Let $x\in X$ be the limit of a sequence $(x_{n})_{n\in\N}$ in $X$. By~\eqref{co:1}, we have $\lim_{n\uparrow\infty} \Lambda(s,x_{n},x) = 0$ for all $s\in I$, and from the description of convergence in $p$-th mean in~\cite[Theorem~21.4 and Corollary~21.5]{Bau01} we obtain that $\lim_{n\uparrow\infty}\int_{I(t)}\!\lambda(t,s)^{p}\Lambda(s,x_{n},x)^{p}\,\mu(\mathrm{d}s)$ $= 0$ for each $t\in I$. From this we conclude that $\lim_{n\uparrow\infty} \Psi(x_{n}) = \Psi(x)$.
\end{proof}

\subsection{Proofs of the results on kernels and resolvent sequences}

\begin{proof}[Proof of Lemma~\ref{le:sufficient criterion}]
There is $n_{t}\in\N$ such that $\int_{I(t)}\!\R_{k^{p},\mu,n}(t,s)u_{0}(s)^{p}\,\mu(\mathrm{d}s)$ is bounded by $\|k_{n}(t,\cdot)\|_{p_{n},t}\|l_{n}(t,\cdot)u_{0}^{p}\|_{q_{n},t}$ for each $n\in\N$ with $n\geq n_{t}$, by H\"{o}lder's inequality. This estimate implies the claims of the lemma.
\end{proof}

\begin{proof}[Proof of Lemma~\ref{le:resolvents of sums of kernels}]
We show the two assertions by induction. In the initial case $n = 1$ the definition $\R_{k,\mu,1,j} = k_{j} = \R_{k_{j},\mu,1}$ for all $j\in\{1,\dots,N\}$ verifies the claims. So, let us suppose that the two claims hold for some $n\in\N$.

Then we have $\R_{k,\mu,n+1,(i,\dots,i)}(t,s) = \int_{[s,t]}\!k_{i}(t,\tilde{s})\R_{k,\mu,n,(i,\dots,i)}(\tilde{s},s)\,\mu(\mathrm{d}\tilde{s}) = \R_{k_{i},\mu,n+1}(t,s)$ and
\begin{equation*}
\R_{k,\mu,n+1}(t,s) = \sum_{j_{n+1}\in\{1,\dots,N\}}\sum_{j\in\{1,\dots,N\}^{n}}\int_{[s,t]}\!k_{j_{n+1}}(t,\tilde{s})\R_{k,\mu,n,(j_{1},\dots,j_{n})}(\tilde{s},s)\,\mu(\mathrm{d}\tilde{s}),
\end{equation*}
which equals $\sum_{j\in\{1,\dots,N\}^{n+1}}\R_{k,\mu,n + 1,j}(t,s)$, for all $s,t\in I$ with $s\leq t$, by the recursive definitions~\eqref{eq:resolvent sequence} and~\eqref{eq:resolvents of sums of kernels}. Thus, the induction proof is complete.
\end{proof}

\begin{proof}[Proof of Proposition~\ref{pr:family of functions}]
(i) We verify all the assertions inductively. In the initial case $n=1$ the representation $f_{1,j}(x,y) = x^{\alpha_{j} -1}y^{-\beta_{j}}$ for all $j\in\{1,\dots,N\}$ and $x,y > 0$ directly shows that $f_{1,j}$ is $]0,\infty[$-valued, decreasing in the second variable and continuous and satisfies~\eqref{eq:family of functions 2} for any $x,y > 0$, even with an equal sign.

Next, let us assume that the assertions are valid for some $n\in\N$. Then for each $j\in\{1,\dots,N\}^{n+1}$ the function $g_{n+1,j}:]0,\infty[\times ]0,\infty[\times ]0,1[\rightarrow ]0,\infty[$ defined via
\begin{equation*}
g_{n+1,j}(x,y,\lambda) := (1-\lambda)^{\alpha_{j_{n+1}}-1}(\lambda x + y)^{-\beta_{j_{n+1}}}f_{n,(j_{1},\dots,j_{n})}(\lambda x,y)
\end{equation*}
is decreasing in the second variable and continuous. Hence, the recursive definition~\eqref{eq:family of functions 1} and Fubini's theorem entail that $f_{n+1,j}$ is $]0,\infty]$-valued, decreasing in the second variable and measurable. Moreover,
\begin{align*}
g_{n+1,j}(x,y,\lambda) &\leq c_{n,(j_{1},\dots,j_{n})}(1-\lambda)^{\alpha_{j_{n+1}}-1}(\lambda x + y)^{-\beta_{j_{n+1}}}(\lambda x)^{\beta_{j_{1}} - 1 + \sum_{i=1}^{n} \alpha_{j_{i}} - \beta_{j_{i}}}y^{-\beta_{j_{1}}}\\
&\leq c_{n,(j_{1},\dots,j_{n})}x^{-1 + \sum_{i=1}^{n}\alpha_{j_{i}} - \beta_{j_{i+1}}}y^{-\beta_{j_{1}}}(1-\lambda)^{\alpha_{j_{n+1}}-1}\lambda^{-1 + \sum_{i=1}^{n}\alpha_{j_{i}} - \beta_{j_{i+1}}}
\end{align*}
for all $x,y > 0$ and $\lambda\in ]0,1[$. Thereby, the first and the second inequalities turn into identities if $\beta_{j_{1}} = \cdots = \beta_{j_{n}} = 0$ and $\beta_{j_{n+1}} = 0$, respectively. By dominated convergence, this implies that $f_{n+1,j}$ is sequentially continuous, and we obtain that
\begin{align*}
&f_{n+1,j}(x,y) = x^{\alpha_{j_{n+1}}}\int_{0}^{1}\!g_{n+1,j}(x,y,\lambda)\,\mathrm{d}\lambda\\
&\leq c_{n,(j_{1},\dots,j_{n})} x^{\alpha_{j_{n+1}}}y^{-\beta_{j_{1}}}\int_{0}^{1}\!(1-\lambda)^{\alpha_{j_{n+1}}-1}(\lambda x + y)^{-\beta_{j_{n+1}}}(\lambda x)^{\beta_{j_{1}} - 1 + \sum_{i=1}^{n} \alpha_{j_{i}} - \beta_{j_{i}}}\,\mathrm{d}\lambda\\
&\leq c_{n,(j_{1},\dots,j_{n})} \mathrm{B}\bigg(\sum_{i=1}^{n}\alpha_{j_{i}} - \beta_{j_{i+1}},\alpha_{j_{n+1}}\bigg) x^{\beta_{j_{1}} - 1 + \sum_{i=1}^{n+1} \alpha_{j_{i}} - \beta_{j_{i}}}y^{-\beta_{j_{1}}}
\end{align*}
for any $x,y > 0$ with equality if $\beta_{j_{1}} = \cdots = \beta_{j_{n+1}} = 0$. As the definition of $c_{n+1,j}$ ensures that $c_{n,(j_{1},\dots,j_{n})}\mathrm{B}(\sum_{i=1}^{n}\alpha_{j_{i}} - \beta_{j_{i+1}},\alpha_{j_{n+1}}) = c_{n+1,j}$, the induction proof is complete.

(ii)We have $\alpha_{j_{1}} - \beta_{j_{2}} + \cdots + \alpha_{j_{i}} - \beta_{j_{i+1}} \geq (\alpha_{0} - \beta_{\infty})i$ for all $i,n\in\N$ with $i + 1\leq n$ and $j\in\{1,\dots,N\}^{n}$. Thus, as $\mathrm{B}$ is decreasing in both variables, we obtain that
\begin{equation*}
c_{n,j} \leq \prod_{i=1}^{n-1}\mathrm{B}\big((\alpha_{0} - \beta_{\infty})i,\alpha_{0}\big) = \frac{\hat{c}_{n}\Gamma(\alpha_{0})^{n}}{\Gamma((\alpha_{0} - \beta_{\infty})n + \beta_{\infty})}
\end{equation*}
for any $n\in\N$, and equality holds if $\alpha_{1} = \cdots = \alpha_{N}$ and $\beta_{1} = \cdots = \beta_{N}$. So, the first two assertions are verified.

Regarding the third claim, we note that $\hat{c}_{n} = \hat{c}_{n_{\Gamma}}\prod_{i = n_{\Gamma}}^{n-1}\frac{\Gamma((\alpha_{0} - \beta_{\infty})i)}{\Gamma((\alpha_{0} - \beta_{\infty})i + \beta_{\infty})} \leq \hat{c}_{n_{\Gamma}}$ for each $n\in\N$ with $n > n_{\Gamma}$, since  $x_{\Gamma} \leq (\alpha_{0} - \beta_{\infty})i \leq (\alpha_{0} - \beta_{\infty})i + \beta_{\infty}$ for every $i\in\N$ with $i\geq n_{\Gamma}$ and $\Gamma$ is strictly increasing on $[x_{\Gamma},\infty[$.

(iii) As pointwise limit of the sequence $(\sum_{i=1}^{n}\sum_{j\in\{1,\dots,N\}^{i}}f_{i,j})_{n\in\N}$ of partial sums, $f$ is $]0,\infty]$-valued, decreasing in the second variable and measurable. Moreover, from the estimates~\eqref{eq:family of functions 2} and~\eqref{eq:family of functions 3} we infer that
\begin{align*}
f_{n,j}(x,y) \leq \frac{\hat{c}_{n}\Gamma(\alpha_{0})^{n}}{\Gamma((\alpha_{0} - \beta_{\infty})n + \beta_{\infty})}\frac{\varphi_{\beta_{\infty},\beta_{0}}(x)}{x\varphi_{\beta_{\infty},\beta_{0}}(y)}\varphi_{\alpha_{0} - \beta_{\infty},\alpha_{\infty} - \beta_{0}}(x)^{n}
\end{align*}
for all $n\in\N$, $j\in\{1,\dots,N\}^{n}$ and $x,y > 0$, and equality holds if $N = 1$ and $\beta = 0$. Hence, the claimed estimate is valid and it turns into an identity in the asserted case. By~\eqref{eq:limit related to the digamma function}, this also implies that the just considered sequence of partial sums converges locally uniformly to $f$, which ensures the continuity of $f$.
\end{proof}

\begin{proof}[Proof of Proposition~\ref{pr:kernels on Cartesian products}]
(i) It suffices to show the first claim, as the representations of $I(t)$ and $[r,t]$ for any $r,t\in I$ with $r\leq t$ follow immediately from the definition of $\leq$. This definition also entails the triangular set $\Delta$ is of the form
\begin{align*}
\Delta &= \big\{(t,s)\in I\times I\,|\, (t_{i},s_{i})\in\Delta_{i}\text{ for all $i\in\{1,\dots,m\}$}\big\} = f^{-1}(\Delta_{1}\times\cdots\times\Delta_{m}),
\end{align*}
where the map $f:I\times I\rightarrow\prod_{i=1}^{m}I_{i}\times I_{i}$ is given by $f(t,s) := (t_{1},s_{1},\dots,t_{m},s_{m})$. By endowing $\prod_{i=1}^{m}I_{i}\times I_{i}$ with the product $\sigma$-field $\otimes_{i=1}^{m}\mathcal{I}_{i}\otimes\mathcal{I}_{i}$, it is readily checked that $f$ is product measurable.

(ii) We prove the estimate~\eqref{eq:kernels on Cartesian products 2} for all $s,t\in I$ with $s\leq t$ by induction over $n\in\N$ and directly turn to the induction step, since in the initial case we have $\R_{k,\mu,1} = k$ and $\R_{k_{i},\mu_{i},1} = k_{i}$ for any $i\in\{1,\dots,m\}$. Then the recursion in~\eqref{eq:resolvent sequence} yields that
\begin{align}\nonumber
\R_{k,\mu,n + 1}(t,s) &= \int_{[s,t]}\!k(t,\tilde{s})\R_{k,\mu,n}(\tilde{s},s)\,\mu(\mathrm{d}\tilde{s})\\\label{eq:product representation inequality}
& \leq \int_{[s_{1},t_{1}]\times\cdots\times [s_{m},t_{m}]}\!\prod_{i=1}^{m}k_{i}(t_{i},\tilde{s}_{i})\R_{k_{i},\mu_{i},n}(\tilde{s}_{i},s_{i})\,\mu_{1}\otimes\cdots\otimes\mu_{m}(\mathrm{d}\tilde{s})\\\nonumber
&= \prod_{i=1}^{m}\int_{[s_{i},t_{i}]}\!k_{i}(t_{i},\tilde{s}_{i})\R_{k_{i},\mu_{i},n}(\tilde{s}_{i},s_{i})\,\mu_{i}(\mathrm{d}\tilde{s}_{i}) = \prod_{i=1}^{m}\R_{k_{i},\mu_{i},n+1}(t_{i},s_{i}),
\end{align}
due to the inequalities~\eqref{eq:kernels on Cartesian products 1} and Fubini's theorem. In particular, if the inequalities in~\eqref{eq:kernels on Cartesian products 1} are equations, then we may assume in the induction hypothesis that equality holds in~\eqref{eq:kernels on Cartesian products 2}. So, there is actually an equal sign in~\eqref{eq:product representation inequality} in this case.

In consequence, the estimates~\eqref{eq:kernels on Cartesian products 3} follow for each $t\in I$ from the definition of $\mathrm{I}_{k,\mu,p}(t)$ in~\eqref{eq:series function} by employing the bound~\eqref{eq:kernels on Cartesian products 2} for all $n\in\N$ and $s\in I(t)$. In addition, the first estimate in~\eqref{eq:kernels on Cartesian products 3} becomes an identity whenever~\eqref{eq:kernels on Cartesian products 1} are equations.
\end{proof}

\subsection{Proofs of the Gronwall inequalities}

\begin{proof}[Proof of Proposition~\ref{pr:estimates for Gronwall inequalities}]
First, the fact that~\eqref{eq:estimates for Gronwall inequalities 1} turns into an identity for $m = 0$ follows directly from Example~\ref{ex:void preorder}. In the case $m\geq 1$, we recall from Example~\ref{ex:regular kernels on intervals} that
\begin{equation*}
\R_{k_{i}^{p},\mu_{i},n}(t_{i},s_{i}) \leq \frac{k_{i}(t_{i},s_{i})^{p}}{(n-1)!}\bigg(\int_{s_{i}}^{t_{i}}\!k_{i}(t_{i},\tilde{s}_{i})^{p}\,\mu(\mathrm{d}\tilde{s}_{i})\bigg)^{n-1}
\end{equation*}
for all $i\in\{1,\dots,m\}$, $n\in\N$ and $s_{i},t_{i}\in I_{i}$ with $s_{i}\leq t_{i}$, and equality holds whenever the stated specific case in~\eqref{eq:specific case} happens. Further, we immediately see that $\int_{[s,t]}\!k(t,\tilde{s})^{p}\,\mu(\mathrm{d}\tilde{s})$ is finite and agrees with
\begin{equation*}
\int_{I_{m+1}}\!k_{m+1}(\tilde{s}_{m+1})^{p}\,\mu_{m+1}(\mathrm{d}\tilde{s}_{m+1})\prod_{i=1}^{m}\int_{s_{i}}^{t_{i}}\!k_{i}(t_{i},\tilde{s}_{i})^{p}\,\mu_{i}(\mathrm{d}\tilde{s}_{i})\end{equation*}
for any $s,t\in I$ with $s\leq t$, by Fubini's theorem. For this reason, an application of Proposition~\ref{pr:kernels on Cartesian products} shows the asserted estimate~\eqref{eq:estimates for Gronwall inequalities 1} for each $n\in\N$. Namely,
\begin{align*}
\R_{k^{p},\mu,n}(t,s) &= k_{m+1}(s_{m+1})^{p}\bigg(\int_{I_{m+1}}\!k_{m+1}(\tilde{s}_{m+1})^{p}\,\mu_{m+1}(\mathrm{d}\tilde{s}_{m+1})\bigg)^{n-1}\prod_{i=1}^{m}\R_{k_{i}^{p},\mu_{i},n}(t_{i},s_{i})\\
&\leq \frac{k(t,s)^{p}}{((n-1)!)^{m}}\bigg(\int_{[s,t]}\!k(t,\tilde{s})^{p}\,\mu(\mathrm{d}\tilde{s})\bigg)^{n-1}.
\end{align*}

For the proof of the second claim, we readily note that~\eqref{eq:estimates for Gronwall inequalities 2} is a necessary condition for~\eqref{eq:resolvent inequality condition 2} for every $t\in I$, as $\R_{k^{p},\mu,1} = k^{p}$. For the converse direction, we infer from the derived estimate~\eqref{eq:estimates for Gronwall inequalities 1} that
\begin{equation*}
\int_{I(t)}\!\R_{k^{p},\mu,n}(t,s)u_{0}(s)^{p}\,\mu(\mathrm{d}s)\leq\frac{c_{m}(t)^{n-1}}{((n-1)!)^{m}}\int_{I(t)}\!k(t,s)^{p}u_{0}(s)^{p}\,\mu(\mathrm{d}s)
\end{equation*}
for any $n\in\N$ and $t\in I$ with the real constant $c_{m}(t) := \int_{I(t)}k(t,s)^{p}\,\mu(\mathrm{d}s)$, which equals $\int_{I}\!k_{m+1}(s)^{p}\,\mu(\mathrm{d}s)$ for $m = 0$ and
\begin{equation*}
\int_{I_{m+1}}\!k_{m+1}(s_{m+1})^{p}\,\mu_{m+1}(\mathrm{d}s_{m+1})\prod_{i=1}^{m}\int_{I_{i}(t_{i})}\!k_{i}(t_{i},s_{i})^{p}\,\mu_{i}(\mathrm{d}s_{i})
\end{equation*}
for $m\geq 1$, by Fubini's theorem. Hence, if $c_{m}(t)$ or $\int_{I(t)}\!k(t,s)^{p}u_{0}(s)^{p}\,\mu(\mathrm{d}s)$ vanish, then so does $\int_{I(t)}\!\R_{k^{p},\mu,n}(t,s)u_{0}(s)^{p}\,\mu(\mathrm{d}s)$ for all $n\in\N$. Otherwise, the series in~\eqref{eq:resolvent inequality condition 2} is dominated by an absolutely convergent series, as the ratio test shows:
\begin{equation*}
\lim_{n\uparrow\infty} \bigg(\frac{c_{m}(t)^{n}}{(n!)^{m}}\frac{((n-1)!)^{m}}{c_{m}(t)^{n-1}}\bigg)^{\frac{1}{p}}  = \lim_{n\uparrow\infty} \bigg(\frac{c_{m}(t)}{n^{m}}\bigg)^{\frac{1}{p}} =
\begin{cases}
(\int_{I}\!k_{m+1}(s)^{p}\,\mu(\mathrm{d}s))^{\frac{1}{p}} & \text{for $m = 0$}\\
0 & \text{for $m\geq 1$}
\end{cases}
.
\end{equation*}

Finally, since $\int_{I(t)}\!k(t,s)^{p}\,\mu(\mathrm{d}s) < \infty$ for each $t\in I$, the first part of the third claim follows from the second assertion by taking $u_{0} = 1$, and the second part is readily seen.
\end{proof}

\begin{proof}[Proof of Lemma~\ref{le:integral estimates}]
By Minkowski's inequality, we may assume that $v_{0} = 1$ and $l = 0$ or instead $v_{0} = 0$. In the first case, the asserted bound becomes the identity
\begin{equation}\label{eq:integral estimates 1}
\int_{I(t)}\!\frac{k(t,s)^{p}}{((n-1)!)^{m}}\bigg(\int_{[s,t]}\!k(t,\tilde{s})^{p}\,\mu(\mathrm{d}\tilde{s})\bigg)^{n-1}\,\mu(\mathrm{d}s) = \frac{1}{(n!)^{m}}\bigg(\int_{I(t)}\!k(t,s)^{p}\,\mu(\mathrm{d}s)\bigg)^{n}
\end{equation}
for any $n\in\N$ and $t\in I$, which is directly checked for $m = 0$. If instead $m\geq 1$, then we derive equation~\eqref{eq:integral estimates 1} by using Fubini's theorem and the formula
\begin{equation*}
\int_{I_{i}(t_{i})}\!\frac{k_{i}(t_{i},s_{i})^{p}}{(n-1)!}\bigg(\int_{s_{i}}^{t_{i}}\!k_{i}(t_{i},\tilde{s}_{i})^{p}\,\mu_{i}(\mathrm{d}\tilde{s}_{i})\bigg)^{n-1}\,\mu_{i}(\mathrm{d}s_{i}) = \frac{1}{n!}\bigg(\int_{I_{i}(t_{i})}\!k_{i}(t_{i},s_{i})^{p}\,\mu_{i}(\mathrm{d}s_{i})\bigg)^{n}
\end{equation*}
for every $i\in\{1,\dots,m\}$, which follows from the fundamental theorem of calculus for Riemann-Stieltjes integrals and monotone convergence. In the second case when $v_{0} = 0$, the claimed estimate reduces to
\begin{equation}\label{eq:integral estimates 2}
\begin{split}
\int_{I(t)}\!\frac{k(t,s)^{p}}{((n-1)!)^{m}}\bigg(\int_{[s,t]}&\!k(t,\tilde{s})^{p}\,\mu(\mathrm{d}\tilde{s})\bigg)^{n-1}\int_{I(s)}\!l(s,r)^{p}\,\mu(\mathrm{d}r)\,\mu(\mathrm{d}s)\\
&\leq \frac{1}{(n!)^{m}}\int_{I(t)}\!\bigg(\int_{[s,t]}\!k(t,\tilde{s})^{p}\,\mu(\mathrm{d}\tilde{s})\bigg)^{n}l(t,s)^{p}\,\mu(\mathrm{d}s),
\end{split}
\end{equation}
which turns into an equation for $m = 0$, as we readily see. If instead $m\geq 1$, then~\eqref{eq:integral estimates 2} is deduced from Fubini's theorem and the formula
\begin{align*}
\int_{I_{i}(t_{i})}\!\frac{k_{i}(t_{i},s_{i})^{p}}{(n-1)!}\bigg(\int_{s_{i}}^{t_{i}}&\!k_{i}(t_{i},\tilde{s}_{i})^{p}\,\mu_{i}(\mathrm{d}\tilde{s}_{i})\bigg)^{n-1}\int_{I_{i}(s_{i})}\!l_{i}(t_{i},r_{i})^{p}\,\mu_{i}(\mathrm{d}r_{i})\,\mu_{i}(\mathrm{d}s_{i})\\
&= \frac{1}{n!}\int_{I_{i}(t_{i})}\!\bigg(\int_{s_{i}}^{t_{i}}\!k_{i}(t_{i},\tilde{s}_{i})^{p}\,\mu_{i}(\mathrm{d}\tilde{s}_{i})\bigg)^{n}l_{i}(t_{i},s_{i})^{p}\,\mu_{i}(\mathrm{d}s_{i})
\end{align*}
for each $i\in\{1,\dots,m\}$, which we infer from the fundamental theorem of calculus for Riemann-Stieltjes integrals, integration by parts and monotone convergence. This also justifies that if $l_{1},\dots,l_{m}$ satisfy the condition in~(ii), then~\eqref{eq:integral estimates 2} is an identity.
\end{proof}

\begin{proof}[Proof of Corollary~\ref{co:Gronwall sequence inequalities}]
The first claimed estimate, which becomes an identity in the stated cases, follows immediately from Proposition~\ref{pr:resolvent sequence inequality} by employing the estimate~\eqref{eq:estimates for Gronwall inequalities 3} for any $n\in\N$ and $t\in J$. The second estimate is a direct consequence of Lemma~\ref{le:integral estimates}.
\end{proof}

\begin{proof}[Proof of Corollary~\ref{co:Gronwall inequalities}]
The condition~\eqref{eq:resolvent inequality condition 2}, which implies~\eqref{eq:resolvent inequality condition 1}, is satisfied for every $t\in J$, by Proposition~\ref{pr:estimates for Gronwall inequalities}. Thus, we obtain the first asserted estimate, which turns into an identity in the stated cases, from Corollary~\ref{co:resolvent inequality} by using the estimate~\eqref{eq:estimates for Gronwall inequalities 3} for all $n\in\N$ and $t\in J$. The second estimate is implied by Lemma~\ref{le:integral estimates}.

Alternatively, the claimed estimates follow from Corollary~\ref{co:Gronwall sequence inequalities} for the choice $u_{n} = u$ for all $n\in\N$, since the sequence $(w_{n})_{n\in\N}$ of $[0,\infty]$-valued measurable functions on $I$ appearing there converges pointwise to zero.

Moreover, the last two series in~\eqref{eq:Gronwall inequalities} converge absolutely for every $t\in I$, because $\int_{I(t)}\!k(t,s)^{p}\,\mu(\mathrm{d}s)$ and $\int_{I(t)}\!l(t,s)^{p}\,\mu(\mathrm{d}s)$ are finite. In fact, if $\int_{I(t)}\!k(t,s)^{p}\,\mu(\mathrm{d}s) = 0$, then both series vanish. Otherwise, we may apply the ratio test to see that the first series converges absolutely, which entails the absolute convergence of the second.
\end{proof}

\let\OLDthebibliography\thebibliography
\renewcommand\thebibliography[1]{
  \OLDthebibliography{#1}
  \setlength{\parskip}{1pt}
  \setlength{\itemsep}{2pt}}


\begin{thebibliography}{10}

\bibitem{Alm17}
R.~Almeida.
\newblock A {G}ronwall inequality for a general {C}aputo fractional operator.
\newblock {\em Math. Inequal. Appl.}, 20(4):1089--1105, 2017.

\bibitem{Ama90}
H.~Amann.
\newblock {\em Ordinary differential equations}, volume~13 of {\em De Gruyter
  Studies in Mathematics}.
\newblock Walter de Gruyter \& Co., Berlin, 1990.
\newblock An introduction to nonlinear analysis, Translated from the German by
  Gerhard Metzen.

\bibitem{Ban1922}
S.~Banach.
\newblock Sur les op\'erations dans les ensembles abstraits et leur application
  aux \'equations int\'egrales.
\newblock {\em Fund. Math.}, 3:133--181, 1922.

\bibitem{Bau01}
H.~Bauer.
\newblock {\em Measure and integration theory}, volume~26 of {\em De Gruyter
  Studies in Mathematics}.
\newblock Walter de Gruyter \& Co., Berlin, 2001.
\newblock Translated from the German by Robert B. Burckel.

\bibitem{Bel43}
R.~Bellman.
\newblock The stability of solutions of linear differential equations.
\newblock {\em Duke Math. J.}, 10:643--647, 1943.

\bibitem{ChuMet67}
S.~C. Chu and F.~T. Metcalf.
\newblock On {G}ronwall's inequality.
\newblock {\em Proc. Amer. Math. Soc.}, 18:439--440, 1967.

\bibitem{CodLev55}
E.~A. Coddington and N.~Levinson.
\newblock {\em Theory of ordinary differential equations}.
\newblock McGraw-Hill Book Co., Inc., New York-Toronto-London, 1955.

\bibitem{GriLonSta90}
G.~Gripenberg, S.-O. Londen, and O.~Staffans.
\newblock {\em Volterra integral and functional equations}, volume~34 of {\em
  Encyclopedia of Mathematics and its Applications}.
\newblock Cambridge University Press, Cambridge, 1990.

\bibitem{Gro80}
J.~Groh.
\newblock A nonlinear {V}olterra-{S}tieltjes integral equation and a {G}ronwall
  inequality in one dimension.
\newblock {\em Illinois J. Math.}, 24(2):244--263, 1980.

\bibitem{Gro1919}
T.~H. Gronwall.
\newblock Note on the derivatives with respect to a parameter of the solutions
  of a system of differential equations.
\newblock {\em Ann. of Math. (2)}, 20(4):292--296, 1919.

\bibitem{Har02}
P.~Hartman.
\newblock {\em Ordinary differential equations}, volume~38 of {\em Classics in
  Applied Mathematics}.
\newblock Society for Industrial and Applied Mathematics (SIAM), Philadelphia,
  PA, 2002.
\newblock Corrected reprint of the second (1982) edition [Birkh\"auser, Boston,
  MA; MR0658490 (83e:34002)], With a foreword by Peter Bates.

\bibitem{HauMatSax11}
H.~J. Haubold, A.~M. Mathai, and R.~K. Saxena.
\newblock Mittag-{L}effler functions and their applications.
\newblock {\em J. Appl. Math.}, pages Art. ID 298628, 51, 2011.

\bibitem{Hea74}
V.~B. Headley.
\newblock A multidimensional nonlinear {G}ronwall inequality.
\newblock {\em J. Math. Anal. Appl.}, 47:250--255, 1974.

\bibitem{Hel69}
B.~W. Helton.
\newblock A product integral representation for a {G}ronwall inequality.
\newblock {\em Proc. Amer. Math. Soc.}, 23:493--500, 1969.

\bibitem{Hen81}
D.~Henry.
\newblock {\em Geometric theory of semilinear parabolic equations}, volume 840
  of {\em Lecture Notes in Mathematics}.
\newblock Springer-Verlag, Berlin-New York, 1981.

\bibitem{Her69}
J.~V. Herod.
\newblock A {G}ronwall inequality for linear {S}tieltjes integrals.
\newblock {\em Proc. Amer. Math. Soc.}, 23:34--36, 1969.

\bibitem{IkeWat89}
N.~Ikeda and S.~Watanabe.
\newblock {\em Stochastic differential equations and diffusion processes},
  volume~24 of {\em North-Holland Mathematical Library}.
\newblock North-Holland Publishing Co., Amsterdam; Kodansha, Ltd., Tokyo,
  second edition, 1989.

\bibitem{Jon64}
G.~S. Jones.
\newblock Fundamental inequalities for discrete and discontinuous functional
  equations.
\newblock {\em J. Soc. Indust. Appl. Math.}, 12:43--57, 1964.

\bibitem{KalMeyPro24-2}
A.~Kalinin, T.~Meyer-Brandis, and F.~Proske.
\newblock Stability, uniqueness and existence of solutions to
  {M}c{K}ean--{V}lasov {SDE}s: a multidimensional {Y}amada--{W}atanabe
  approach.
\newblock {\em Stoch. Dyn.}, 24(5):Paper No. 2450039, 2024.

\bibitem{KalMeyPro24}
A.~Kalinin, T.~Meyer-Brandis, and F.~Proske.
\newblock Stability, uniqueness and existence of solutions to
  {M}c{K}ean-{V}lasov stochastic differential equations in arbitrary moments.
\newblock {\em J. Theoret. Probab.}, 37(4):2941--2989, 2024.

\bibitem{KarShr91}
I.~Karatzas and S.~E. Shreve.
\newblock {\em Brownian motion and stochastic calculus}, volume 113 of {\em
  Graduate Texts in Mathematics}.
\newblock Springer-Verlag, New York, second edition, 1991.

\bibitem{Lin13}
S.-y. Lin.
\newblock Generalized {G}ronwall inequalities and their applications to
  fractional differential equations.
\newblock {\em J. Inequal. Appl.}, pages 2013:549, 9, 2013.

\bibitem{LiuXuZhou24}
Z.~Liu, X.~Xu, and T.~Zhou.
\newblock Finite-time multistability of a multidirectional associative memory
  neural network with multiple fractional orders based on a generalized
  gronwall inequality.
\newblock {\em Neural Computing and Applications}, pages 1--23, 2024.

\bibitem{Mag23}
R.~Magnus.
\newblock {\em Essential ordinary differential equations}.
\newblock Springer, 2023.

\bibitem{Mao08}
X.~Mao.
\newblock {\em Stochastic differential equations and applications}.
\newblock Horwood Publishing Limited, Chichester, second edition, 2008.

\bibitem{Pac75}
B.~G. Pachpatte.
\newblock On some generalizations of {B}ellman's lemma.
\newblock {\em J. Math. Anal. Appl.}, 51:141--150, 1975.

\bibitem{Pac98}
B.~G. Pachpatte.
\newblock {\em Inequalities for differential and integral equations}, volume
  197 of {\em Mathematics in Science and Engineering}.
\newblock Academic Press, Inc., San Diego, CA, 1998.

\bibitem{Ras76}
D.~L. Rasmussen.
\newblock Gronwall's inequality for functions of two independent variables.
\newblock {\em J. Math. Anal. Appl.}, 55(2):407--417, 1976.

\bibitem{RevYor99}
D.~Revuz and M.~Yor.
\newblock {\em Continuous martingales and {B}rownian motion}, volume 293 of
  {\em Grundlehren der mathematischen Wissenschaften [Fundamental Principles of
  Mathematical Sciences]}.
\newblock Springer-Verlag, Berlin, third edition, 1999.

\bibitem{SchSel68}
W.~W. Schmaedeke and G.~R. Sell.
\newblock The {G}ronwall inequality for modified {S}tieltjes integrals.
\newblock {\em Proc. Amer. Math. Soc.}, 19:1217--1222, 1968.

\bibitem{Sno72}
D.~R. Snow.
\newblock Gronwall's inequality for systems of partial differential equations
  in two independent variables.
\newblock {\em Proc. Amer. Math. Soc.}, 33:46--54, 1972.

\bibitem{VanCap19}
J.~Vanterler~da Costa~Sousa and E.~Capelas~de Oliveira.
\newblock A {G}ronwall inequality and the {C}auchy-type problem by means of
  {$\psi$}-{H}ilfer operator.
\newblock {\em Differ. Equ. Appl.}, 11(1):87--106, 2019.

\bibitem{Web19}
J.~R.~L. Webb.
\newblock Weakly singular {G}ronwall inequalities and applications to
  fractional differential equations.
\newblock {\em J. Math. Anal. Appl.}, 471(1-2):692--711, 2019.

\bibitem{Wil65}
D.~Willett.
\newblock A linear generalization of {G}ronwall's inequality.
\newblock {\em Proc. Amer. Math. Soc.}, 16:774--778, 1965.

\bibitem{WriKlaKen71}
F.~M. Wright, M.~L. Klasi, and D.~R. Kennebeck.
\newblock The {G}ronwall inequality for weighted integrals.
\newblock {\em Proc. Amer. Math. Soc.}, 30:504--510, 1971.

\bibitem{YeGaoDin07}
H.~Ye, J.~Gao, and Y.~Ding.
\newblock A generalized {G}ronwall inequality and its application to a
  fractional differential equation.
\newblock {\em J. Math. Anal. Appl.}, 328(2):1075--1081, 2007.

\bibitem{ZhanWei16}
Z.~Zhang and Z.~Wei.
\newblock A generalized {G}ronwall inequality and its application to fractional
  neutral evolution inclusions.
\newblock {\em J. Inequal. Appl.}, pages Paper No. 45, 18, 2016.

\end{thebibliography}
\end{document}